\documentclass[a4paper,twoside]{article}

\usepackage[dvipsnames]{xcolor}

\usepackage{xspace}
\usepackage[ulem=normalem,draft]{changes}

\usepackage{amsthm,amsmath,amssymb}
\usepackage[bookmarks=true,bookmarksopen=true]{hyperref}
\usepackage{graphicx}
\usepackage{amssymb,amsmath}
\usepackage{url}
\usepackage{xspace}

\usepackage{xspace}

\numberwithin{equation}{section}
\theoremstyle{plain}
\newtheorem{theorem}{Theorem}[section]
\newtheorem{lemma}[theorem]{Lemma}
\newtheorem{corollary}[theorem]{Corollary}
\newtheorem{proposition}[theorem]{Proposition}
\newtheorem{definition}[theorem]{Definition}

\newtheorem{remark}[theorem]{Remark}

\newcommand{\Pb}{\mbox{\rm (P)}\xspace}
\newcommand{\PbT}{\mbox{\rm (P$_T$)}\xspace}
\newcommand{\Uad}{\mathcal{U}_{ad}}
\newcommand{\UadT}{{\mathcal{U}_{ad,T}}}
\newcommand{\K}{\mathcal{K}}
\newcommand{\I}{0,\infty}
\newcommand{\IT}{0,T}

\newcommand{\sign}{\operatorname{sign}}

\newcommand{\dx}{\,\mathrm{d}x}
\newcommand{\dt}{\,\mathrm{d}t}
\newcommand{\ds}{\,\mathrm{d}s}

\newcommand{\black}{\color{black}}

\begin{document}
	
\title{Infinite Horizon Optimal Control Problems for a Class of Semilinear Parabolic Equations
\thanks{The first author was supported by MCIN/ AEI/10.13039/501100011033/ under research project PID2020-114837GB-I00. The second was supported by the ERC advanced grant 668998 (OCLOC) under the EU’s H2020 research program.}}

\author{Eduardo Casas\thanks{Departamento de Matem\'{a}tica Aplicada y Ciencias de la Computaci\'{o}n, E.T.S.I. Industriales y de Telecomunicaci\'on, Universidad de Cantabria, 39005 Santander, Spain (eduardo.casas@unican.es).}
    \and Karl Kunisch\thanks{Institute for Mathematics and Scientific Computing, University of Graz, Heinrichstrasse 36, A-8010 Graz, Austria (karl.kunisch@uni-graz.at).}}

\date{}
\maketitle

\begin{abstract}
Infinite horizon open loop optimal control problems for semilinear parabolic equations  are investigated. The controls are subject to a cost-functional which promotes sparsity in time.  The focus is put on deriving first order optimality conditions. This is achieved without relying on a  well-defined control-to-state mapping in a neighborhood of minimizers. The technique of proof is based on the approximation of the original problem  by a family of finite horizon problems. The  optimality conditions allow to deduce sparsity properties of the optimal controls in time.

\end{abstract}


%%35J61, % semilinear elliptic equations
%35K58,  % semilinear parabolic equations
%49J20, % optimal control problems involving partial differential equations
%49J52, % nonsmooth analysis
%49K20, % optimality conditions: problems involving partial differential equations
%%49M25, discrete approximations
%%90C48, % programming in abstract spaces
%\end{AMS}


\pagestyle{myheadings} \thispagestyle{plain} \markboth{E.~CASAS AND K.~KUNISCH}{Infinite Horizon Control Problems}

\section{Introduction}
\label{S1}

In this paper we continue our investigations of infinite horizon optimal control problems with sparsity promoting cost functionals, which we commenced in \cite{Casas-Kunisch2017}. While in the earlier work the nonlinearities appearing in the state equation were restricted to be polynomials we now allow general nonlinearities satisfying appropriate properties at the origin and asymptotically. Moreover, differently from \cite{Casas-Kunisch2017}, constraints on the controls are imposed.

While finite horizon open loop optimal control problems with partial differential equations as constraints have received a tremendous amount of attention over the last fifty years, extremely little attention was paid to infinite horizon problems, see, however, \cite[Chapter III.6]{JL1971}, and  \cite{BKP2018} for an analysis of  bilinear optimal control problems. This is different for problems involving the control of ordinary differential equations. The analysis of infinite horizon optimal control problems may have started with Halkin's work \cite{halkin1974}. The motivation for investigating infinite horizon problems relates to stabilization problems
as well as to problems arising in the economic and biological sciences,
where placing a finite bound  on the time horizon introduces an artificial ambiguity.
Some examples in economy, biology, and engineering motivating this study, and many aspects of extensive earlier work were described in \cite{CHL1991}. More recent contributions can be found for instance in \cite{AKV17} and \cite{BCF2018}. In passing, let us recall that  the infinite horizon problem is well investigated  for closed loop control, leading to the Riccati synthesis for the linear-quadratic regulator problem and the stationary Hamilton-Jacobi-Bellman equation otherwise.

In the present work we formulate an optimal tracking problem as an infinite horizon optimal control problem. The control cost is  chosen in such a manner that  it enhances sparsity in time. As a consequence the control  will shut down to zero rather than being small, as  it  would be the case for a quadratic cost, for instance. A first difficulty that needs to be addressed is the  existence of feasible controls. This relates to the stabilizability problem. While this is not in the focus of the present work we establish certain sufficient conditions where stabilizability holds. The central difficulty then is to provide first order optimality conditions. They are easily conjectured but challenging to verify. Here we follow the technique of formulating a family of finite horizon problems and analyze the asymptotic behavior as the horizon tends to infinity.

We now introduce the optimal control problem which will be analyzed in the present work:
\[
   \Pb \qquad  \min_{u \in \Uad} J(u) = \frac{1}{2}\int_0^\infty\int_\Omega (y_u - y_d)^2\dx\dt + \kappa\int_0^\infty\Big(\int_\omega u^2\ dx\Big)^{1/2}\dt,
\]
where $\kappa > 0$, $y_d \in L^2(\Omega \times (\I))$, and
\[
\Uad = \{u \in L^\infty(0,\infty;L^2(\omega)) : u(t) \in \K \ \text{for a.a. } t \in (\I)\}.
\]
Above $\K$ denotes a closed, convex, and bounded set in $L^2(\Omega)$, and $y_u$ is the solution of the following parabolic equation:
\begin{equation}
\left\{
\begin{array}{l}
\displaystyle
\frac{\partial y}{\partial t}- \Delta y + ay + f(x,t,y) = g + u\chi_\omega \mbox{ in } Q = \Omega \times (\I),\\ \partial_ny = 0  \mbox{ on } \Sigma = \Gamma \times (\I),\ y(0) = y_0  \mbox{ in } \Omega.
\end{array}
\right.
\label{E1.1}
\end{equation}
Here \black $\Omega$ is a bounded domain in $\mathbb{R}^n$, $1 \le n \le 3$, with a Lipschitz boundary $\Gamma$, $\omega$ is a measurable subset of $\Omega$ with positive Lebesgue measure, $\chi_\omega$ denotes the characteristic function of $\omega$, $a \in L^\infty(\Omega)$, $0 \le a \not\equiv 0$, $g \in L^2(Q) \cap L^\infty(\I;L^2(\Omega))$, and $y_0 \in L^\infty(\Omega)$. The conditions on the nonlinear term $f(x,t,y)$ will be given below. For every $u \in \Uad$, the symbol $u\chi_\omega$ is defined as follows:
\[
(u\chi_\omega)(x,t) = \left\{\begin{array}{cl}u(x,t) & \text{if } (x,t) \in  Q_\omega = \omega \times (\I),\\0 & \text{otherwise.}\end{array}\right.
\]

Possible choices for $\K$ include
\begin{align}
&\K = B_\gamma = \{v \in L^2(\omega) : \|v\|_{L^2(\omega)} \le \gamma\}, \ 0 < \gamma < \infty,\label{E1.2}\\
&\K = \{v \in L^2(\omega) : \alpha \le v(x) \le \beta \text{ for a.a. } x \in \omega\},\ \alpha < 0 < \beta.\label{E1.3}
\end{align}

In this paper, all the results remain valid if we replace the operator $-\Delta$ and the normal derivative $\partial_n$ by a more general elliptic operator $Ay = -\sum_{i, j = 1}^n\partial_{x_j}[a_{ij}(x)\partial_{x_i}y]$ and its associated normal derivative $\partial_{n_A}$ with  the coefficients $a_{ij} \in L^\infty(\Omega)$ satisfying the usual ellipticity condition.

The contents of the paper are structured as follows. Section 2 contains an analysis of the state equation and existence of a solution for \Pb. The auxiliary finite horizon problems and their optimality systems are presented in Section 3. Section 4 contains the convergence analysis of the finite horizon problems. An optimality system for the original problem can then be deduced in Section 5. Its interpretation allows to derive the sparsity in time of the optimal controls.

It is worth pointing out that these conditions are obtained without relying on a well-defined control-to-state mapping in an open  neighborhood of  optimal controls. In the presence of the high generality that we allow for our nonlinearity $f$, at present we do not know whether such a neighborhood exists.

\section{Well-Posedness of the State Equation and Problem \Pb}
\label{S2}
\setcounter{equation}{0}

We define the notion of solution for \eqref{E1.1}. First, let us fix some notation. We denote by $L_{loc}^2(\I;H^1(\Omega))$  the space of functions $y$ belonging to $L^2(\IT;H^1(\Omega))$ for $0 < T < \infty$. Analogously we define $L^p_{loc}(\I;L^2(\Omega))$ for $1 \le p \le \infty$.

\begin{definition}
We call $y$ a solution to \eqref{E1.1} if $y \in L_{loc}^2(\I;H^1(\Omega))$, and for every $T > 0$ the restriction of $y$ to $Q_T = \Omega \times (\IT)$ belongs to $W(0,T) \cap L^\infty(Q_T)$ and satisfies the following equation in the variational sense
\begin{equation}
\left\{
\begin{array}{l}
\displaystyle\frac{\partial y}{\partial t}- \Delta y + ay + f(x,t,y) = g + u\chi_\omega \mbox{ in } Q_T,\\ \partial_ny = 0 \mbox{ on } \Sigma_T,\ y(0) = y_0  \mbox{ in } \Omega.
\end{array}
\right.
\label{E2.1}
\end{equation}
\label{D2.1}
\end{definition}

Here $W(0,T)$ denotes the space of functions $y \in L^2(0,T;H^1(\Omega))$ such that $\frac{\partial y}{\partial t} \in L^2(0,T;H^1(\Omega)^*)$.

In order to prove the existence and uniqueness of a solution to \eqref{E1.1} we make the following assumptions:  $f:Q \times \mathbb{R} \longrightarrow \mathbb{R}$ is a Carath\'eodory function of class $C^1$ with respect to the last variable satisfying
\begin{align}
&f(x,t,0) = 0,\label{E2.2}\\
&\left\{\hspace{-0.2cm}\begin{array}{l}\exists M_f > 0, \, \exists \delta > 0,\, \text{and a } C^1\, \text{function}\, \hat f:\mathbb{R} \longrightarrow \mathbb{R} \text{ such that } \forall |y| \ge M_f\\  \delta |\hat f(y)| \le |f(x,t,y)| \le |\hat f(y)|,\, \sign{\hat f(y)} = \sign{f(x,t,y)} = \sign{y},\, \hat f'(y) \ge 0,\\\displaystyle\frac{\partial f}{\partial y}(x,t,y) \ge 0,\end{array}\right.\label{E2.3}\\
&\forall M > 0 \ \exists C_M \text{ such that } \left|\frac{\partial f}{\partial y}(x,t,y)\right| \le C_M \ \forall |y| \le M,\label{E2.4}
\end{align}
for almost every $(x,t) \in Q$.

Let us give some examples that fulfilled the assumptions \eqref{E2.2}--\eqref{E2.4}. We start with a polynomial function of $y$ with coefficients depending on $(x,t)$:
\[
f(x,t,y) = \sum_{k = 1}^{2m + 1}a_k(x,t)y^k  \text{ with }  a_k \in L^\infty(Q) \ \forall k \ge 1  \text{ and } a_{2m + 1}(x,t) \ge \delta_0 > 0 \text{ in } Q.
\]
Setting $K = \max_{1 \le k \le 2m+1}\|a_k\|_{L^\infty(Q)}$, $M_f = \max\Big\{1,\frac{4mK}{\delta_0}\Big\}$,  $\hat f(y) = (2m+1)Ky^{2m+1}$ ,
and $\delta = \frac{\delta_0}{2(2m+1)K}$, it is easy to check that \eqref{E2.3} holds.

Next, given $\eta \in L^\infty(Q)$ such that $\eta(x,t) \ge \delta_0 > 0$ for a.a.~$(x,t) \in Q$, we consider the following two examples
\[
f(x,t,y) =\eta(x,t)(\text{\rm e}^y - 1)\ \text{ and } \ f(x,t,y) = \eta(x,t)(y^3 + 10^3\sin(y)).
\]
For the first case \eqref{E2.3} holds with $M_f = 0$, $\delta = \frac{\delta_0}{\|\eta\|_{L^\infty(Q)}}$, and $\hat f(y) = \|\eta\|_{L^\infty(Q)}(\text{\rm e}^y - 1)$.

\noindent For the second case we take $M_f = 10\sqrt{\frac{10}{3}}$, $\delta = \frac{\delta_0}{4\|\eta\|_{L^\infty(Q)}}$, and $\hat f(y) = 2\|\eta\|_{L^\infty(Q)}y^3$.

\begin{remark}
\noindent 1) Assumptions \eqref{E2.2} and  \eqref{E2.3} can be replaced by the following
\[
\left\{\hspace{-0.2cm}\begin{array}{l}\exists M_f > 0, \, \exists \delta > 0,\, \text{and a } C^1\, \text{function}\, \hat f:\mathbb{R} \longrightarrow \mathbb{R} \text{ such that } \forall |y| \ge M_f\\  \delta|\hat f(y)| \le |f(x,t,y) - f(x,t,0)| \le |\hat f(y)|,\ \hat f'(y) \ge 0,\\
\sign{\hat f(y)} = \sign{y},\, y[f(x,t,y) - f(x,t,0)] \ge 0, \displaystyle\frac{\partial f}{\partial y}(x,t,y) \ge 0,\end{array}\right.
\]for almost every $(x,t) \in Q$. Indeed, under these hypotheses we can replace $f(x,t,y)$ by $f(x,t,y) - f(x,t,0)$ and $g(x,t)$ by $g(x,t) - f(x,t,0)$ so that the new $f$ satisfies \eqref{E2.2}--\eqref{E2.4}, and the new $g$ belongs to $L^2(Q)$.

\noindent 2) Let us observe that \eqref{E2.2}--\eqref{E2.4} imply that
\begin{equation}
\frac{\partial f}{\partial y}(x,t,y) \ge -C_{M_f} \text{ for a.a. } (x,t) \in Q \text{ and } \forall y \in \mathbb{R}.
\label{E2.5}
\end{equation}

\noindent 3) If $y \in L^\infty(Q_T)$, then \eqref{E2.2} and \eqref{E2.4} imply that $f(\cdot,\cdot,y) \in L^\infty(Q_T)$ as well.
\label{R2.1}
\end{remark}

\begin{remark}
From the assumption $0 \le a \not\equiv 0$ we infer the existence of a constant $C_a$ depending on $a$ such that
\begin{equation}
\Big(\int_\Omega[|\nabla y|^2 + ay^2]\dx\Big)^{1/2} \ge C_a\|y\|_{H^1(\Omega)}\ \forall y \in H^1(\Omega).
\label{E2.6}
\end{equation}
For homogeneous \black Dirichlet boundary condition, the assumption $a \not\equiv 0$ is not required.
\label{R2.2}
\end{remark}

\begin{theorem}
Under the previous assumptions on $f$, equation \eqref{E1.1} has a unique solution $y$ for every  $u \in L^2(Q_\omega) \cap L^\infty(\I;L^2(\omega))$\black . Moreover, if $y \in L^2(Q)$ then the following properties hold
\begin{align}
&\exists\, C_f \text{ such that } \|f(\cdot,\cdot,y)\|_{L^2(Q)} \le C_f\Big(\|u\|_{L^2(Q_\omega)} + \|g\|_{L^2(Q)} + \|y\|_{L^2(Q)}\Big),\label{E2.7}\\
&y \in L^2(\I;H^1(\Omega)) \cap C([0,\infty),L^2(\Omega)) \text{ and } \frac{\partial y}{\partial t} \in L^2(\I,H^1(\Omega)^*),\label{E2.8}\\
&\lim_{t \to \infty}\|y(t)\|_{L^2(\Omega)} = 0.\label{E2.9}
\end{align}
\label{T2.1}
\end{theorem}

\begin{proof}
Due to  $u \in L^\infty(\I;L^2(\omega))$ and $g \in L^\infty(\I;L^2(\Omega))$, under the assumptions on $f$ and inequality \eqref{E2.5}, the proof of existence and uniqueness of a solution $y \in W(0,T) \cap L^\infty(Q_T)$ for \eqref{E2.1} for every $T>0$ is standard. This proves the first statement of the theorem. The rest of the proof, under the assumption $y \in L^2(Q)$, is divided into three steps.

{\em Step 1 - $f(\cdot,\cdot,y) \in L^2(Q)$.} Given $M = \max\{M_f,\|y_0\|_{L^\infty(\Omega)}\}$, we define the following sets:
\begin{align}
&Q^M = \{(x,t) \in Q : |y(x,t)| > M\},\ \ Q^M_T = Q_T \cap Q^M,\label{E2.10}\\
&\Omega_t^M = \{x \in \Omega : (x,t) \in Q^M\}\text{ for } t \in (0,\infty). \label{E2.11}
\end{align}
Using \eqref{E2.2}, \eqref{E2.4}, and the mean value theorem we have for $\theta(x,t) \in (0,1)$
\begin{equation}
\int_{Q \setminus Q^M}|f(x,t,y)|^2\dx\dt = \int_{Q \setminus Q^M}\Big|\frac{\partial f}{\partial y}(x,t,\theta y)\Big|^2y^2\dx\dt \le C_M^2\|y\|^2_{L^2(Q)}.
\label{E2.12}
\end{equation}
It remains to prove that $\int_{Q_T^M}|f(x,t,y)|^2\dx\dt$ is uniformly bounded with respect to $T$. For this purpose we consider the decomposition $Q_T = Q_T^{M,+} \cup Q_T^{M,-}$, where
\[
Q_T^{M,+} = \{(x,t) \in Q_T^M : y(x,t) > M\} \text{ and } Q_T^{M,-} = \{(x,t) \in Q_T^M : y(x,t) < -M\}.
\]
We will prove that $\int_{Q_T^{M,+}}|f(x,t,y)|^2\dx\dt \le C$ for some constant $C$ independent of $T > 0$. Similar arguments can be applied to prove the uniform boundedness on the sets $Q_T^{M,-}$.  We define the function $f_M:\mathbb{R} \longrightarrow \mathbb{R}$ by $f_M(s) = \hat f(\max\{s,M\}) - \hat f(M)$.  Then $f_M$ is locally Lipschitz and $f_M(s) = 0$ for $s \le M$. Furthermore, from \eqref{E2.4} we get that $f_M(s) \ge 0$ for $s \ge M$ and $f'_M(s) \ge 0$ for almost all $s \in \mathbb{R}$.

We also introduce the $C^1$ function $F_M:\mathbb{R} \longrightarrow \mathbb{R}$ given by $F_M(s) = \int_{-\infty}^sf_M(t)\dt$. Then, we have that $F_M(s) = 0$ for every $s \le M$. In particular, due to the choice of $M$ we have that $F_M(y_0(x)) = 0$ for almost all $x \in \Omega$. Since $f_M(s) \ge 0$ $\forall s\in \mathbb{R}$, we also have that $F_M(s) \ge 0$ for every $s \in \mathbb{R}$.

Due to the embedding $W(0,T) \subset C([0,T];L^2(\Omega))$, we find that $\lim_{t \to 0}y(t) = y_0$ in $L^2(\Omega)$. Hence, we can take a sequence of points $\{t_k\}_{k = 1}^\infty \subset (0,T)\black$ converging to 0 such that $y(x,t_k) \to y_0(x)$ for almost all $x \in \Omega$. We set $z(x,t) =f_M(y(x,t))$. Since $y \in L^\infty(Q_T) \cap H^1(\Omega \times (t_k,T))$ for every $k$ (see, for instance, \cite[Corollary III.2.4]{Showalter1997}) and $f_M$ is locally Lipschitz, we deduce that $z \in L^\infty(Q_T) \cap H^1(\Omega \times (t_k,T))$ as well. Testing the equation \eqref{E2.1} with $z$ and integrating in $\Omega \times (t_k,T)$ we infer
\begin{align}
&\int_{t_k}^T\int_\Omega\frac{\partial y}{\partial t}z\dx\dt + \int_{t_k}^T\int_\Omega[\nabla y\nabla z + ayz]\dx\dt\notag\\
& + \int_{t_k}^T\int_\Omega f(x,t,y)z\dx\dt = \int_{t_k}^T\int_\Omega (g + \chi_\omega u)z\dx\dt. \label{E2.13}
\end{align}
We study  the first two terms of the above identity. From the definition of $F_M$ we get
\[
\int_{t_k}^T\int_\Omega\frac{\partial y}{\partial t}z\dx\dt = \int_{t_k}^T\frac{d}{\dt}\int_\Omega F_M(y) \dx\dt = \int_\Omega F_M(y(T)) \dx - \int_\Omega F_M(y(t_k)) \dx.
\]
Taking limits in this equality when $k \to \infty$ and using that $F_M(y_0) = 0$ \black we infer
\begin{equation}
\lim_{k \to \infty}\int_{t_k}^T\int_\Omega\frac{\partial y}{\partial t}z\dx\dt = \int_\Omega F_M(y(T)) \dx \ge 0.
\label{E2.14}
\end{equation}

To deal with the second term we observe that $z(x,t) = 0$ if $x \not\in \Omega_t^{M,+} = \{x \in \Omega : (x,t) \in Q^{M,+}_T\}$. Using  that $y(x,t) \ge M$ and $z(x,t) \ge 0$ in $\Omega_t^{M,+}$ we infer with \eqref{E2.3}
\begin{align}
&\lim_{k \to \infty}\int_{t_k}^T\int_\Omega[\nabla y\nabla z + ayz]\dx\dt = \int_0^T\int_\Omega[\nabla y\nabla z + ayz]\dx\dt\notag\\
&=\int_0^T\int_{\Omega_t^{M,+}}[\hat f'(y)|\nabla y|^2 + ayz]\dx\dt \ge 0. \label{E2.15}
\end{align}

Passing to the limit in\eqref{E2.13} as $k \to \infty$,  we deduce with \eqref{E2.14}, \eqref{E2.15},  and \eqref{E2.3}
\begin{align}
&\delta\int_0^T\int_{\Omega_t^{M,+}} \hat f(y)[\hat f(y) - \hat f(M)]\dx\dt = \delta\int_0^T\int_\Omega \hat f(y)z\dx\dt  \le \int_0^T\int_\Omega f(x,t,y)z\dx\dt\notag\\
&\le \int_0^T\int_\Omega (g + \chi_\omega u)z\dx\dt = \int_0^T\int_{\Omega_t^{M,+}} (g + \chi_\omega u)z\dx\dt.
\label{E2.16}
\end{align}
Observe that \black
\begin{equation}
|Q^M| \le \frac{1}{M^2}\int_{Q^M}y^2\dx\dt \le \frac{1}{M^2}\|y\|^2_{L^2(Q)}.
\label{E2.17}
\end{equation}

Let us set $\Omega_t^{M,+} = A_t \cup B_t$ with
\[
A_t = \{x \in \Omega_t^{M,+} : \hat f(y(x,t)) < 2\hat f(M)\} \text{ and } B_t = \{x \in \Omega_t^{M,+} : \hat f(y(x,t)) \ge 2\hat f(M)\}.
\]

If $x \in B_t$, then $-\hat f(M) \ge -\frac{1}{2}\hat f(y(x,t))$ and, consequently, $\hat f(y(x,t)) - \hat f(M) \ge \frac{1}{2}\hat f(y(x,t))$ holds. This yields
\begin{equation}
\frac{1}{2}\int_0^T\int_{B_t}\hat f(y)^2\dx\dt \le \int_0^T\int_{B_t}\hat f(y)[\hat f(y) - \hat f(M)]\dx\dt.
\label{E2.18}
\end{equation}
Using \eqref{E2.18} and then \eqref{E2.16} and the fact that $\hat f(y)[\hat f(y) - \hat f(M)] \ge 0$ in $A_t$ we infer \black
\begin{align*}
&\frac{\delta}{2}\int_0^T\int_{B_t}\hat f(y)^2\dx\dt \le \delta\int_0^T\int_{B_t}\hat f(y)[\hat f(y) - \hat f(M)]\dx\dt\\
&= \delta\int_0^T\int_{\Omega^{M,+}_t}\hat f(y)[\hat f(y) - \hat f(M)]\dx\dt - \delta\int_0^T\int_{A_t}\hat f(y)[\hat f(y) - \hat f(M)]\dx\dt\\
&\le \int_0^T\int_{\Omega_t^{M,+}}(g + \chi_\omega u)z\dx\dt \le \int_0^T\int_{A_t}(g + \chi_\omega u)z\dx\dt + \int_0^T\int_{B_t}(g + \chi_\omega u)z\dx\dt\\
& \le \|g + \chi_\omega u\|_{L^2(Q)}\left[\left(\int_0^T\int_{A_t}z^2\dx\dt\right)^{1/2} + \left(\int_0^T\int_{B_t}z^2\dx\dt\right)^{1/2}\right]\\
&\le \hat f(M)\sqrt{|Q^M|}\|g + \chi_\omega u\|_{L^2(Q)} + \frac{1}{\delta}\|g + \chi_\omega u\|^2_{L^2(Q)} + \frac{\delta}{4}\int_0^T\int_{B_t}z^2\dx\dt\\
&\le \big(\frac{1}{2} + \frac{1}{\delta}\big)\|g + \chi_\omega u\|^2_{L^2(Q)} + \frac{\hat f(M)^2}{2M^2}\|y\|^2_{L^2(Q)} + \frac{\delta}{4}\int_0^T\int_{B_t}\hat f(y)^2\dx\dt.
\end{align*}
From here we obtain
\[
\int_0^T\int_{B_t}\hat f(y)^2\dx\dt \le \big(\frac{2}{\delta} + \frac{4}{\delta^2}\big) \|g + \chi_\omega u\|^2_{L^2(Q)} + \frac{2\hat f(M)^2}{\delta M^2}\|y\|^2_{L^2(Q)}.
\]
In $A_t$, with \eqref{E2.17} we deduce the estimate
\[
\int_0^T\int_{A_t}\hat f(y)^2\dx\dt \le 4\hat f(M)^2|Q^M| \le \frac{4\hat f(M)^2}{M^2}\|y\|^2_{L^2(Q)}.
\]
These two estimates and the analogous ones in $Q_T^{M,-}$  along  with \eqref{E2.3} and \eqref{E2.12} lead to
\begin{equation}
\|f(\cdot,\cdot,y)\|_{L^2(Q)} \le \|\hat f(y)\|_{L^2(Q)} \le C_f\Big(\|u\|_{L^2(Q_\omega)} + \|g\|_{L^2(Q)} + \|y\|_{L^2(Q)}\Big).
\label{E2.19}
\end{equation}

{\em Step 2. Proof of \eqref{E2.8}.} For every $t > 0$ we have with \eqref{E2.1}
\begin{align*}
&\frac{1}{2}\|y(t)\|^2_{L^2(\Omega)} + C_a^2\int_0^t\|y\|^2_{H^1(\Omega)}\dt \le \frac{1}{2}\|y(t)\|^2_{L^2(\Omega)} + \int_0^t\int_\Omega[|\nabla y|^2 + ay^2]\dx\ds\\
& = \int_0^t\int_\Omega[g + \chi_\omega u - f(x,t,y)]y\dx\ds + \frac{1}{2}\|y_0\|^2_{L^2(\Omega)}\\
& \le \|g + \chi_\omega u - f(\cdot,\cdot,y)\|_{L^2(Q)}\Big(\int_0^t\|y\|^2_{H^1(\Omega)}\ds\Big)^{1/2} + \frac{1}{2}\|y_0\|^2_{L^2(\Omega)}\\
&\le \frac{1}{2C_a^2}\|g + \chi_\omega u - f(\cdot,\cdot,y)\|^2_{L^2(Q)} + \frac{C_a^2}{2}\int_0^t\|y\|^2_{H^1(\Omega)}\ds + \frac{1}{2}\|y_0\|^2_{L^2(\Omega)},
\end{align*}
what implies
\[
\|y(t)\|^2_{L^2(\Omega)} + C_a^2\int_0^t\|y\|^2_{H^1(\Omega)}\dt \le \frac{1}{C_a^2}\|g + \chi_\omega u - f(\cdot,\cdot,y)\|^2_{L^2(Q)} + \|y_0\|^2_{L^2(\Omega)}.
\]
Using that $y \in W(0,T) \subset C([0,T];L^2(\Omega))$ for every $T < \infty$ along with the above inequality we infer that $y \in L^2(\I;H^1(\Omega)) \cap C([0,\infty);L^2(\Omega))$. Finally, we have
\[
\frac{\partial y}{\partial t} = g + \chi_\omega u - f(\cdot,\cdot,y) + \Delta y - ay \in L^2(\I;H^1(\Omega)^*).
\]

{\em Step 3. Proof of \eqref{E2.9}.}  The fact that $y \in L^2(Q)$ implies the existence of a monotone increasing sequence of positive numbers $\{t_k\}_{k = 1}^\infty$ converging to $\infty$ such that $\|y(t_k)\|_{L^2(\Omega)} \to 0$ as $k \to \infty$. Given $T > 0$ and taking $t_k > T$ we get
\begin{align*}
&\|y(T)\|^2_{L^2(\Omega)} = \|y(t_k)\|^2_{L^2(\Omega)} - 2\int_T^{t_k}\int_\Omega \langle\frac{\partial y}{\partial t}(t),y(t)\rangle\dt\\
& \le \|y(t_k)\|^2_{L^2(\Omega)} + 2\|y\|_{L^2(T,\infty;H^1(\Omega))}\|\frac{\partial y}{\partial t}\|_{L^2(\I;H^1(\Omega)^*)},
\end{align*}
where $\langle\cdot,\cdot\rangle$ denotes the duality between $H^1(\Omega)^*$ and $H^1(\Omega)$. Taking the limit when $k \to \infty$ we get
\[
\|y(T)\|^2_{L^2(\Omega)} \le 2\|y\|_{L^2(T,\infty;L^2(\Omega))}\|\frac{\partial y}{\partial t}\|_{L^2(\I;H^1(\Omega)^*)}.
\]
Passing to the limit when $T \to \infty$, \eqref{E2.9} follows.
\end{proof}

Now, we analyze the existence of a solution for \Pb. First, we observe that Theorem \ref{T2.1} guaranties that  $J:L^\infty(\I;L^2(\omega)) \cap L^2(Q_\omega) \longrightarrow [0,\infty]$ is well defined. We  introduce the following functions
\begin{equation}
\left\{\hspace{-0.25cm}\begin{array}{l}
j:L^1(\I;L^2(\omega)) \longrightarrow \mathbb{R} \text{ and } j_T:L^1(\IT;L^2(\omega)) \longrightarrow \mathbb{R} \text{ by}\\
j(u) = \|u\|_{L^1(\I;L^2(\omega))} \text{ and } j_T(u) = \|u\|_{L^1(\IT;L^2(\omega))}.\end{array}\right.
\label{E2.20}
\end{equation}
%and
%\begin{equation}
%\left\{\hspace{-0.25cm}\begin{array}{l}
%J:L^\infty(\I;L^2(\omega)) \longrightarrow [0,\infty] \text{ and } J_T:L^\infty(\IT;L^2(\omega)) \longrightarrow \mathbb{R} \text{ by}\\
%J(u) = \frac{1}{2}\|y_u - y_d\|^2_{L^2(Q)} + \kappa j(u) \text{ and } J_T(u) = \frac{1}{2}\|y_u - y_d\|^2_{L^2(Q_T)} + \kappa j_T(u).\end{array}\right.
%\label{E2.22}
%\end{equation}

\begin{theorem}
If there exists a control $\hat u \in \Uad \cap L^1(\I;L^2(\omega))$ such that its associated state $\hat y \in L^2(Q)$, then \Pb has at least one solution $\bar u$.
\label{T2.2}
\end{theorem}

\begin{proof}
Let $\{u_k\}_{k = 1}^\infty \subset \Uad$ be a minimizing sequence for \Pb with associated states $\{y_k\}_{k = 1}^\infty$. Then, for $k$ big enough we have that $J(u_k) \le J(\hat u)$, unless $\hat u$ is already a solution of \Pb. As a consequence of this inequalities and the control constraint we get that $\{u_k\}_{k = 1}^\infty$ and $\{y_k\}_{k = 1}^\infty$ are bounded in $L^\infty(\I;L^2(\Omega)) \cap L^1(\I;L^2(\omega))$ and $L^2(Q)$, respectively. As $L^\infty(\I;L^2(\Omega)) \cap L^1(\I;L^2(\omega)) \subset L^2(Q_\omega)$, we take subsequences, denoted in the same form, such that $u_k \rightharpoonup \bar u$ in $L^2(Q_\omega)$, $u_k \stackrel{*}{\rightharpoonup} \bar u$ in $L^\infty(\I;L^2(\omega))$, and $y_k \rightharpoonup \bar y$ in $L^2(Q)$. Since $\{{y_k}_{\mid Q_T}\}_{k = 1}^\infty$ is bounded in $W(0,T) \cap L^\infty(Q_T)$ for every $T > 0$, we can pass to the limit in equation \eqref{E2.1} satisfied by $(u_k,y_k)$ to deduce that $(\bar u,\bar y)$ satisfies \eqref{E2.1} for all $T$. Hence, $\bar y$ is the state associated to $\bar u$.

Since $u_k \rightharpoonup \bar u$ in $L^2(Q_\omega)$ and $\Uad \cap L^2(Q_\omega)$ is convex and closed in $L^2(Q_\omega)$, we infer that $\bar u \in \Uad$. Moreover,  the restrictions of the functionals $j_T$ to $L^2(Q_T)$ are convex and continuous, and thus $j_T(\bar u) \le \liminf_{k \to \infty}j_T(u_k)$. Therefore, we obtain
\begin{align*}
\frac{1}{2}\|\bar y - y_d\|^2_{L^2(Q_T)} + \kappa j_T(\bar u) &\le \liminf_{k \to \infty}\Big\{\frac{1}{2}\|y_k - y_d\|^2_{L^2(Q_T)} + \kappa j_T(u_k)\Big\}\\
& \le \liminf_{k \to \infty}J(u_k) = \inf\Pb,
\end{align*}
and, consequently,
\[
J(\bar u) = \sup_{T > 0}\Big\{\frac{1}{2}\|\bar y - y_d\|^2_{L^2(Q_T)} + \kappa j_T(\bar u)\Big\} \le \inf\Pb.
\]
Therefore, $\bar u$ is a solution of \Pb.
\end{proof}

We end this section by describing special situations in which the existence of an admissible control as demanded in Theorem \ref{T2.2}  is guaranteed. At first we give a sufficient condition on $f$ such  that  for all sufficiently small $y_0$ the zero-control is admissible.  The embedding constant from $H^1(\Omega)$ into $L^4(\Omega)$ will be denoted by $C_4$.

\begin{proposition}
Assume that $f$ satisfies \eqref{E2.2}--\eqref{E2.4} as well as
\begin{equation*}
\exists\, m_f \in (0,M_f) \text{ such that } \frac{\partial f}{\partial y}(x,t,y) \ge 0 \text{ if } |y| < m_f.
\end{equation*}
Then, for each $y_0$ with $\|y_0\|_{L^2(\Omega)} < K_f = \frac{m_f C_a^4}{(C_{M_f}C_4)^2}$ there exists $\lambda > 0$ such that  the solution $y$ to \eqref{E2.2} with $u=0$ satisfies
$$
\|y\|_{L^2(Q} \le \frac{1}{\sqrt \lambda}\Big( \|y_0\|_{L^2(\Omega)} + \|g\|_{L^2(Q)}\Big).
$$
\end{proposition}
\begin{proof} To verify this result we can follow the proof of \cite[Theorem 2.5]{Casas-Kunisch2017}. For this purpose we use that  our conditions on $f$ imply that $y(x,t) f(x,t, y(x,t)) \ge 0$ for almost every $(x,t)\in Q$ satisfying $|y(x,t)| < m_f$ or $|y(x,t)| > M_f$. We set $K_0= \frac{1}{2} (\|y_0\|_{L^2(\Omega)} + K_f)<K_f $. Following the above mentioned proof we obtain
$$
\frac{1}{2} \frac{d}{dt} \|y(t)\|^2_{L^2(\Omega)} + \lambda \|y(t)\|^2_{L^2(\Omega)} \le (g(t),y(t))_{L^2(\Omega)}.
$$
where $\lambda = \frac{1}{2}\big(C_a^2 - \frac{C^2_{M_f}C^2_4 K_0}{m_f C_a^2 } \big)>0$. From here we deduce that
$$
\frac{d}{dt} \|y(t)\|^2_{L^2(\Omega)} + \lambda \|y(t)\|^2_{L^2(\Omega)} \le \frac{1}{\lambda} \|g(t)\|^2_{L^2(\Omega)}.
$$
Multiplying by $\text{exp}(\lambda s)$ and integrating on $(0,t)$ we find
$$
e^{\lambda t}  \|y(t)\|^2_{L^2(\Omega)} \le \|y(0)\|^2_{L^2(\Omega)} + \frac{1}{\lambda} \int_0^t e^{\lambda s} \|g(s)\|^2_{L^2(\Omega)} ds.
$$
Multiplying by $\text{exp}(-\lambda t)$ and integrating on $(0,\infty)$ we obtain the desired estimate
\[
\|y\|_{L^2(Q)} \le \frac{1}{\sqrt \lambda}\Big[\|y_0\|_{L^2(\Omega)} + \Big(\int_0^\infty \|g(t)\|^2_{L^2(\Omega)} dt\Big)^{1/2}\Big].
\]
\end{proof}

For the following result we assume that $f$ is independent of $(x,t)$ and that $f'(0) >0$. Then there exists $\rho^+\in(0,\infty]$ such that $f$ has no change of sign in $(0,\rho^+)$, and analogously there exists $\rho^-\in[-\infty,0)$ with  no sign change  in $(\rho^-,0)$.
\begin{proposition}
Assume that $f$ satisfies \eqref{E2.2}--\eqref{E2.4} as well as $f'(0)>0$, and that $g=0$. Then for each $\rho_-\le y_0\le \rho^+$ the solution $y$ for $u=0$ belongs $L^2(Q)$ and $\|y(t)\|_{L^2(\Omega)} \le e^{-C_a t} \|y_0\|_{L^2(\Omega)}$ for all $t\ge0 $.
\end{proposition}

This can be verified by adapting the proof of \cite[Theorem 2.6]{Casas-Kunisch2017}. A special case of $f$ is given by  the Schl\"ogl model, where $f(x,t,y)= y(y - \xi_1)(y-\xi_2)$ with $\xi_i\in \mathbb{R}$, $i\in \{1,2\}$. It was investigated in  \cite{AKR2021}
with  finite dimensional controls of the form $u(t)= \sum_{i=1}^{M} u_i(t) \chi_{\omega_i^M}$, where $\omega^M_i\subset \Omega$, and $\sum_{i=1}^{M} |\omega_i^M| = r |\Omega|$ for some fixed $r\in(0,1)$ independent of $M$.  From these results it follows as a special case that for $g=0, y_d=0$, $\Omega$ a hypercube, and  for each $\lambda>0$ there exist $M$, $\gamma>0$,  and $\{\omega_i\}_{i=1}^{M}$ such that the controlled system decays exponentially with rate $-\lambda$, and that the controls, which are constructed in feedback form, are admissible. This holds globally for all $y_0 \in H^1(\Omega)$.

\section{Auxiliary Finite Horizon Problems}
\label{S3}
\setcounter{equation}{0}

We denote by $\bar u$ a solution of \Pb with associated state $\bar y$. In order to establish an optimality system satisfied by $(\bar u,\bar y)$ we introduce some auxiliary finite horizon control problems. For every $T > 0$ we consider the problem
\begin{align*}
& \PbT \min_{u \in \UadT} J_T(u) = F_T(u) + \kappa j_T(u),\\
&\text{\rm where } F_T(u) = \frac{1}{2}\|y_{T,u} - y_d\|^2_{L^2(Q_T)} +  \frac{1}{2}\|y_{T,u} - \bar y\|^2_{L^2(Q_T)},\ \ j_T(u) = \|u\|_{L^1(0,T;L^2(\omega))},\\
&\UadT = \{u \in L^\infty(\IT;L^2(\omega)) : u(t) \in \K \ \text{for a.a. } t \in (\IT)\},
\end{align*}
and $y_{T,u}$ is the solution of \eqref{E2.1} corresponding to $u$. In the following theorem $\partial j_T$ denotes the convex subdifferential of the function $j_T$ and $Q_{T,\omega} = \omega \times (0,T)$.

\begin{theorem}
Problem \PbT has at least one solution $u_T$ with associated state $y_T$. Furthermore, there exist $\varphi_T \in W(0,T)$, $\lambda_T \in \partial j_T(u_T) \subset L^\infty(\IT;L^2(\omega))$, and $\mu_T \in L^\infty(\IT;L^2(\omega))$ such that
\begin{align}
&\left\{\begin{array}{l}
\displaystyle
-\frac{\partial\varphi_T}{\partial t}- \Delta\varphi_T + a\varphi_T + \frac{\partial f}{\partial y}(x,t,y_T)\varphi_T =  2y_T - y_d - \bar y \mbox{ in } Q_T,\\[0.5ex]
\partial_n\varphi_T  = 0  \mbox{ on } \Sigma_T,\ \varphi_T(T) = 0 \mbox{ in } \Omega,
\end{array}\right.
\label{E3.1}\\
&\int_0^T\int_\omega\mu_T(u - u_T)\dx\dt \le 0\quad \forall u \in \UadT,\label{E3.2}\\
&{\varphi_T}_{\mid Q_{T,\omega}} + \kappa\lambda_T + \mu_T = 0. \label{E3.3}
\end{align}
\label{T3.1}
\end{theorem}

\begin{proof}
The proof for existence of a solution $u_T$ is the same as in Theorem \ref{T2.2}. Let us derive the optimality conditions. It is well known that $F_T:L^\infty(\IT;L^2(\omega)) \longrightarrow \mathbb{R}$ is of class $C^1$ and its derivative at $u_T$ is given by
\[
F'_T(u_T)v = \int_{Q_{T,\omega}}\varphi_{T,u}v\dx\dt,
\]
where $\varphi_{T,u} \in W(0,T)$ is the solution of \eqref{E3.1}. Since $\varphi_T \in W(0,T) \subset C([\IT];L^2(\Omega))$ the linear form $F'_T(u_T):L^\infty(\IT;L^2(\omega)) \longrightarrow \mathbb{R}$ can be  extended to a continuous form $F'_T(u_T):L^1(\IT;L^2(\omega)) \longrightarrow \mathbb{R}$. Now, using the optimality of $u_T$ and the convexity of $j_T$ we derive for every $u \in \UadT$ and every $\rho \in (0,1)$ small enough
\begin{align*}
&\frac{F_T(u_T + \rho(u - u_T)) - F_T(u_T)}{\rho} + \kappa[j_T(u) - j_T(u_T)]\\
&\ge \frac{J_T(u_T + \rho(u - u_T)) - J_T(u_T)}{\rho} \ge 0.
\end{align*}
Then, passing to the limit when $\rho \to 0$ we infer
\[
\int_{Q_{T,\omega}}\varphi_T(u - u_T)\dx\dt + \kappa[j_T(u) - j_T(u_T)] \ge 0 \quad \forall u \in \UadT.
\]
If we denote by $I_{\UadT}:L^1(\IT;L^2(\omega)) \longrightarrow [0,\infty]$ the indicator function taking the value 0 if $u \in \UadT$ and $\infty$ otherwise, then the above inequality implies that $u_T$ is a solution for the optimization problem
\[
\min_{u \in L^1(\IT;L^2(\omega))} \int_{Q_{T,\omega}}\varphi_T u\dx\dt + \kappa j_T(u) + I_{\UadT}(u).
\]
From the subdifferential calculus we deduce that
\[
0 \in {\varphi_T}_{\mid Q_{T,\omega}} + \kappa\partial j_T(u_T) + \partial I_{\UadT}(u_T).
\]
Hence, there exist $\lambda_T \in \partial j_T(u_T)$ and $\mu_T \in \partial I_\UadT(u_T)$ such that \eqref{E3.3} holds. Further, $\mu_T \in \partial I_\UadT(u_T)$ is equivalent to \eqref{E3.2}, which completes the proof.
\end{proof}

Let us observe that $\lambda_T \in \partial j_T(u_T)$ implies that
\begin{equation}
\left\{\begin{array}{l}\|\lambda_T(t)\|_{L^2(\omega)} \le 1 \text{ for a.a. } t \in (\IT), \\[1.5ex]\displaystyle\lambda_T(x,t) = \frac{u_T(x,t)}{\|u_T(t)\|_{L^2(\omega)}} \text{ for a.a. } (x,t) \in Q_{T,\omega} \text{ if } \|u_T(t)\|_{L^2(\omega)} \neq 0;\end{array}\right.
\label{E3.5}
\end{equation}
see, for instance, \cite[Proposition 3.8]{CHW2017}.

Regarding $\mu_T$ we consider separately the two cases specified in \eqref{E1.2} and \eqref{E1.3}.

\begin{lemma}
If $\K$ is chosen as in \eqref{E1.2}, then $\mu_T$ has the following properties
\begin{align}
&\int_\omega\mu_T(t)(v - u_T(t))\dx \le 0 \ \ \forall v \in B_\gamma, \label{E3.6}
\\&\text{If } \|u_T(t)\|_{L^2(\omega)} < \gamma, \text{ then } \|\mu_T(t)\|_{L^2(\omega)} = 0, \label{E3.7}\\
&\text{If }  \|\mu_T(t)\|_{L^2(\omega)} \neq 0, \text{ then } u_T(x,t) = \gamma\frac{\mu_T(x,t)}{\|\mu_T(t)\|_{L^2(\omega)}},\label{E3.8}
\end{align}
for almost all $t \in (\IT)$.
\label{L3.1}
\end{lemma}

\begin{proof}
For the proof of \eqref{E3.6} the reader is referred to the first statement of \cite[Corollary 3.1]{Casas-Kunisch2022}. To prove \eqref{E3.7} and \eqref{E3.8} we proceed as follows
\begin{align*}
&\|\mu_T(t)\|_{L^2(\omega)} = \frac{1}{\gamma}\sup_{v \in B_\gamma}\int_\omega\mu_T(t)v\dx\\
&\le \frac{1}{\gamma}\int_\omega\mu_T(t)u_T(t)\dx \le \frac{1}{\gamma}\|\mu_T(t)\|_{L^2(\omega)}\|u_T(t)\|_{L^2(\omega)}.
\end{align*}
If $\|u_T(t)\|_{L^2(\omega)} < \gamma$, the above inequality is only possible if $\|\mu_T(t)\|_{L^2(\omega)} = 0$, which proves \eqref{E3.7}. Otherwise  we get the equality
\[
\int_\omega\mu_T(t)u_T(t)\dx = \|\mu_T(t)\|_{L^2(\omega)}\|u_T(t)\|_{L^2(\omega)}.
\]
This only possible if there exists a number $g(t) \ge 0$ such that $u_T(x,t) = g(t)\mu_T(x,t)$ for almost all $x \in \omega$. But, we have
\[
\gamma = \|u_T(t)\|_{L^2(\omega)}=g(t)\|\mu_T(t)\|_{L^2(\omega)},
\]
therefore $g(t) = \frac{\gamma}{\|\mu_T(t)\|_{L^2(\omega)}}$, which implies \eqref{E3.8}.
\end{proof}

\begin{lemma}
If $\K$ is chosen as in \eqref{E1.3}, then $\mu_T$ satisfies
\begin{equation}
\mu_T(x,t) \left\{\begin{array}{ll}\le 0 &\text{if } u_T(x,t) = \alpha,\\
= 0 &\text{if } \alpha < u_T(x,t) < \beta,\\ \ge 0 &\text{if } u_T(x,t) = \beta,\end{array}\right.
\label{E3.9}
\end{equation}
for almost all $(x,t) \in \omega \times (0,T)$.
\label{L3.2}
\end{lemma}

This a well known property following from the fact that $\mu_ \in \partial I_\UadT(u_T)$.

\section{Convergence of the Auxiliary Problems}
\label{S4}
\setcounter{equation}{0}

In this section,  $u_T$ denotes a solution of \PbT for $T > 0$. Associated with $u_T$ we have the corresponding state $y_T$ and elements $\varphi_T \in W(0,T)$, $\lambda_T \in \partial j_T(u_T)$, and $\mu_T \in L^\infty(\IT;L^2(\omega))$ satisfying the optimality conditions \eqref{E3.1}--\eqref{E3.3}. We extend all these elements by 0 for $t > T$ and denote these extensions by $(\bar u_T,\bar y_T,\bar\varphi_T,\bar\lambda_T,\bar\mu_T)$.
The analysis of the convergence of $(\bar u_T,\bar y_T)$, $\bar\varphi_T$, and $(\bar\lambda_T,\bar\mu_T)$ is split into three subsections.

\subsection{Convergence of the controls and associated states}
\label{S4.1}

We have a first convergence result.

\begin{theorem}
The following convergence properties hold:
\begin{align}
&\bar y_T \to \bar y \text{ in } L^2(Q),\label{E4.1}\\
&\bar u_T \stackrel{*}{\rightharpoonup} \bar u \text{ in } L^\infty(\I;L^2(\omega)) \cap L^2(Q_\omega),\label{E4.2}\\
&\|\bar u_T\|_{L^1(\I;L^2(\omega))} \to \|\bar u\|_{L^1(\I;L^2(\omega))}. \label{E4.3}
\end{align}
\label{T4.1}
\end{theorem}

\begin{proof}
Using the optimality of $\bar u_T$ we get
\begin{align}
&\frac{1}{2}\|\bar y_T - y_d\|^2_{L^2(Q)} + \frac{1}{2}\|\bar y_T - \bar y\|^2_{L^2(Q)} + \kappa j(\bar u_T)\notag\\
& = J_T(\bar u_T) + \frac{1}{2}\|y_d\|^2_{L^2(T,\infty;L^2(\Omega))} + \frac{1}{2}\|\bar y\|^2_{L^2(T,\infty;L^2(\Omega))}\notag\\
&\le J_T(\bar u) + \frac{1}{2}\|y_d\|^2_{L^2(T,\infty;L^2(\Omega))} + \frac{1}{2}\|\bar y\|^2_{L^2(T,\infty;L^2(\Omega))}\notag\\
& \le J(\bar u) + \frac{1}{2}\|y_d\|^2_{L^2(Q)} + \frac{1}{2}\|\bar y\|^2_{L^2(Q)}. \label{E4.4}
\end{align}
This implies the boundedness of $\{\bar y_T\}_{T > 0}$ and $\{\bar u_T\}_{T > 0}$ in $L^2(Q)$ and $L^1(\I;L^2(\omega))$, respectively. Additionally, since $\bar u_T \in \Uad$ for every $T > 0$, we have that $\{\bar u_T\}_{T > 0}$ is also bounded in $L^\infty(\I;L^2(\omega))$. As a consequence, $\{\bar u_T\}_{T > 0}$ is also bounded in $L^2(Q_\omega)$. Therefore, there exist sequences $\{T_k\}_{k = 1}^\infty$ converging to $\infty$ such that
\begin{equation}
\bar y_{T_k} \rightharpoonup \tilde y \text{ in } L^2(Q) \text{ and } \bar u_{T_k} \stackrel{*}{\rightharpoonup} \tilde u \text{ in } L^\infty(\I;L^2(\omega)) \cap L^2(Q_\omega) \text{ as  } k \to \infty\label{E4.5}
\end{equation}
with $(\tilde u,\tilde y) \in \Uad \times L^2(Q)$. Let us fix $T > 0$. Then, for every $k$ large enough we have that $(\bar u_{T_K},\bar y_{T_k})$ satisfies the equation \eqref{E2.1} in $Q_T$. Using the boundedness of $\{\bar u_{T_k}\}_{T_k \ge T}$ in $L^\infty(\IT,L^2(\omega))$ we infer that $\{\bar y_{T_k}\}_{T_k \ge T}$ is bounded in $W(0,T) \cap L^\infty(Q_T)$. Then, it is easy to pass to the limit in \eqref{E2.1} and deduce that $(\tilde u,\tilde y) \in L^2(Q_T) \times W(0,T) \cap L^\infty(Q_T)$ satisfies the equation in $Q_T$. Consequently, $\tilde y$ is the solution of the state equation \eqref{E1.1} associated with $\tilde u$. To prove that $\tilde u = \bar u$ we first observe that $\bar u_{T_k} \rightharpoonup \tilde u$ in $L^2(Q_T)$ for every $T > 0$. This implies that
\[
j_T(\tilde u) \le \liminf_{k \to \infty}j_T(\bar u_{T_k}) \le \liminf_{k \to \infty}j(\bar u_{T_k}).
\]
Therefore, we have $j(\tilde u) = \sup_{T > 0}j_T(\tilde u) \le \liminf_{k \to \infty}j(\bar u_{T_k})$. Using this property and the fact that $\bar y_{T_k} \rightharpoonup \tilde y$ in $L^2(Q)$ we infer
\begin{align*}
&J(\tilde u) + \frac{1}{2}\|\tilde y - \bar y\|^2_{L^2(Q)} \le \liminf_{k \to \infty}\left\{\frac{1}{2}\|\bar y_{T_k} - y_d\|^2_{L^2(Q)} + \frac{1}{2}\|\bar y_{T_k} - \bar y\|^2_{L^2(Q)} + \kappa j(\bar u_{T_k})\right\}\\
&=\liminf_{k \to \infty}J_{T_k}(u_{T_k}) + \frac{1}{2}\lim_{k \to \infty}\left\{\|y_d\|^2_{L^2(T_k,\infty;L^2(\Omega))} + \|\bar y\|^2_{L^2(T_k,\infty;L^2(\Omega))}\right\}\\
&\le \liminf_{k \to \infty}J_{T_k}(\bar u) = J(\bar u) \le J(\tilde u).
\end{align*}
We have used the optimality of $u_{T_k}$ and $\bar u$ in (P$_{T_k}$) and (P), respectively,  and the fact that $\tilde u$ is a feasible control for \Pb. From the above inequalities we obtain that $\tilde y = \bar y$ and, hence, the state equation yields $\tilde u = \bar u$.  The uniqueness of the limit implies that \eqref{E4.2} and the weak convergence $y_T \rightharpoonup \bar y$ in $L^2(Q)$ hold.  Arguing similarly as above we obtain
\begin{align*}
&J(\bar u) \le \liminf_{T \to \infty}\left\{\frac{1}{2}\|\bar y_T - y_d\|^2_{L^2(Q)} + \kappa j(\bar u_T)\right\}\le \limsup_{T \to \infty}\left\{\frac{1}{2}\|\bar y_T - y_d\|^2_{L^2(Q)} + \kappa j(\bar u_T)\right\}\\
& \le \limsup_{T \to \infty}J_T(\bar u_T) + \frac{1}{2}\lim_{T \to \infty}\|y_d\|^2_{L^2(T,\infty;L^2(\Omega))} \le \limsup_{T \to \infty}J_T(\bar u) = J(\bar u).
\end{align*}
The previous inequalities yield
\[
\lim_{T \to \infty}\left\{\frac{1}{2}\|\bar y_T - y_d\|^2_{L^2(Q)} + \kappa \|\bar u_T\|_{L^1(\I;L^2(\omega))}\right\} = \frac{1}{2}\|\bar y - y_d\|^2_{L^2(Q)} + \kappa \|\bar u\|_{L^1(\I;L^2(\omega))}.
\]
From the Lemma \ref{L4.1} below we infer that
\[
\lim_{T \to \infty}\|\bar y_T - y_d\|^2_{L^2(Q)} = \|\bar y - y_d\|^2_{L^2(Q)} \text{ and } \lim_{T \to \infty}\|\bar u_T\|_{L^1(\I;L^2(\omega))} = \|\bar u\|_{L^1(\I;L^2(\omega))}.
\]
Finally, \eqref{E4.1} and \eqref{E4.3} are straightforward consequences of these limits.
\end{proof}

\begin{lemma}
Let $\{a_T\}_{T > 0}$ and $\{b_T\}_{T > 0}$ two families of real numbers satisfying
\[
a \le \liminf_{T \to \infty}a_T, \ b \le \liminf_{T \to \infty}b_T, \text{ and } \lim_{T \to \infty}(a_T + b_T) = a + b
\]
for $a, b \in \mathbb{R}$. Then, we have that $\lim_{T \to \infty}a_T = a$ and $\lim_{T \to \infty} b_T = b $.
\label{L4.1}
\end{lemma}

\begin{proof}
We prove that $a_T \to a$ as follows
\[
a \le \liminf_{T \to \infty} a_T \le \limsup_{T \to \infty} a_T \le \limsup_{T \to \infty} (a_T + b_T) - \liminf_{T \to \infty} b_T \le (a + b) - b = a.
\]
Now, the convergence of $\{b_T\}_{T > 0}$ is immediate.
\end{proof}

The next theorem establishes stronger convergence properties of $\{\bar y_T\}_{T > 0}$ to $\bar y$.

\begin{theorem}
The following convergences hold
\begin{align}
&f(\cdot,\cdot,\bar y_T) \rightharpoonup f(\cdot,\cdot,\bar y) \text{ in } L^2(Q),\label{E4.6}\\
&\bar y_T \to \bar y \text{ in } L^2(\I;H^1(\Omega)), \label{E4.7}\\
&\lim_{T \to \infty}\|\bar y_T(T)\|_{L^2(\Omega)} = 0. \label{E4.8}
\end{align}
\label{T4.2}
\end{theorem}

\begin{proof}
Arguing as in the proof of \eqref{E2.7} we get
\[
\|f(\cdot,\cdot,\bar y_T)\|_{L^2(Q)} \le C_f\big(\|\bar u_T\|_{L^2(Q_\omega)} + \|g\|_{L^2(Q)} + \|\bar y_T\|_{L^2(Q)}).
\]
From Theorem \ref{T4.1} we deduce that the right hand side of the above inequality is uniformly bounded in $T$. Hence, there exist subsequences $\{f(\cdot,\cdot,\bar y_{T_k})\}_{k = 1}^\infty$ with $T_k \to \infty$ such that $f(\cdot,\cdot,\bar y_{T_k}) \rightharpoonup \psi$ in $L^2(Q)$. Due to the strong convergence $\bar y_{T_k} \to \bar y$ in $L^2(Q)$, we can extract a subsequence, denoted in the same way, such that $\bar y_{T_k}(x,t) \to \bar y(x,t)$ for almost all $(x,t) \in Q$. Then, the continuity of $f$ with respect to $y$ implies that $f(x,t,\bar y_{T_k}(x,t)) \to f(x,t,\bar y(x,t))$ almost everywhere. Hence, $\psi = f(\cdot,\cdot,\bar y)$ and the whole family $\{f(\cdot,\cdot,\bar y_T)\}_{T > 0}$ converges weakly in $L^2(Q)$ to $f(\cdot,\cdot,\bar y)$.

Let us prove that $\bar y_T \rightharpoonup \bar y$ in $L^2(\I;H^1(\Omega))$. From the equation satisfied by $\bar y_T$ we get  $\frac{\partial\bar y_T}{\partial t} - \Delta\bar y_T + a\bar y_T = g + \bar u_T\chi_\omega - f(\cdot,\cdot,\bar y_T)$ in $Q_T$.  Since the right hand side is uniformly bounded in $L^2(Q)$, we deduce the existence of a constant $C$ such that $\|\bar y_T\|_{W(0,T)} \le C$ for every $T > 0$. Therefore, the estimate $\|\bar y_T\|_{L^2(\IT;H^1(\Omega))} \le C$ holds for every $T$. Due to the fact that $\bar y_T = 0$ in $Q \setminus Q_T$, we conclude that $\{\bar y_T\}_{T > 0}$ is bounded in $L^2(\I;H^1(\Omega))$. Then, the convergence $\bar y_T \to \bar y$ in $L^2(Q)$ implies that $\bar y_T \rightharpoonup \bar y$ in $L^2(\I;H^1(\Omega))$. To prove \eqref{E4.7} and \eqref{E4.8} we argue, using \eqref{E4.1} and \eqref{E4.2}, as follows
\begin{align*}
&\int_Q[|\nabla\bar y|^2 + a\bar y^2]\dx\dt \le \liminf_{T \to \infty}\int_Q[|\nabla\bar y_T|^2 + a\bar y_T^2]\dx\dt\\
&\le \limsup_{T \to \infty}\int_Q[|\nabla\bar y_T|^2 + a\bar y_T^2]\dx\dt\\
&\le \limsup_{T \to \infty}\left\{\frac{1}{2}\|\bar y_T(T)\|^2_{L^2(\Omega)} + \int_Q[|\nabla\bar y_T|^2 + a\bar y_T^2]\dx\dt\right\}\\
&=\limsup_{T \to \infty}\left\{\int_Q(g + \bar u_T\chi_\omega)\bar y_T\dx\dt - \int_Qf(x,t,\bar y_T)\bar y_T\dx\dt + \frac{1}{2}\|y_0\|^2_{L^2(\Omega)}\right\}\\
& = \int_Q(g + \bar u\chi_\omega)\bar y \dx\dt - \int_Qf(x,t,\bar y)\bar y\dx\dt + \frac{1}{2}\|y_0\|^2_{L^2(\Omega)} = \int_Q[|\nabla\bar y|^2 + a\bar y^2]\dx\dt.
\end{align*}
These inequalities yield
\[
\lim_{T \to \infty}\|\bar y_T\|_{L^2(\I;H^1(\Omega))} = \|\bar y\|_{L^2(\I;H^1(\Omega))} \ \text{ and }\ \lim_{T \to \infty}\|\bar y_T(T)\|_{L^2(\Omega)} = 0,
\]
which  concludes the proof.
\end{proof}

\begin{theorem}
For every $\varepsilon > 0$ there exists $T_\varepsilon > 0$ such that
\begin{equation}
\|\bar y_T(t)\|_{L^2(\Omega)} < \varepsilon\quad \forall T > T_\varepsilon \text{ and } \forall t > T_\varepsilon.
\label{E4.9}
\end{equation}
\label{T4.3}
\end{theorem}
\begin{proof}
In the proof of Theorem \ref{T4.2}, the existence of a constant $C$ such that $\|\bar y_T\|_{W(0,T)} \le C$ for every $T$ was established. Hence, we have that
\[
\big\|\frac{\partial\bar y_T}{\partial t}\big\|_{L^2(\IT,H^1(\Omega)^*)} \le C\quad \forall T > 0.
\]
By Theorem \ref{T4.2}, for every $\varepsilon >0$ there exists $T_\varepsilon > 0$ such that for every $T > T_\varepsilon$ the inequalities
\[
\|\bar y_T(T)\|_{L^2(\Omega)} < \frac{\varepsilon}{ \sqrt{3}},\ \|\bar y - \bar y_T\|_{L^2(\I;H^1(\Omega))} < \frac{\varepsilon^2}{6C}, \ \|\bar y\|_{L^2(T_\varepsilon,\infty;H^1(\Omega))} < \frac{\varepsilon^2}{6C}
\]
hold. Let us take $T > T_\varepsilon$ arbitrary. For $t = T$, \eqref{E4.9} follows from the choice of $\varepsilon$. If $t > T$, then we have $\bar y_T(t) = 0$ and \eqref{E4.9} holds. Let us take $t \in (T_\varepsilon,T)$, then
\begin{align*}
&\|\bar y_T(t)\|^2_{L^2(\Omega)} = \|\bar y_T(T)\|^2_{L^2(\Omega)} - 2\int_t^T\langle\frac{\partial\bar y_T}{\partial t},\bar y_T\rangle\dt\\
&\le \|\bar y_T(T)\|^2_{L^2(\Omega)} + 2 \|\frac{\partial\bar y_T}{\partial t}\|_{L^2(\IT;H^1(\Omega)^*)}\|\bar y_T\|_{L^2(t,T;H^1(\Omega))}\\
&\le \|\bar y_T(T)\|^2_{L^2(\Omega)} + 2C\Big(\|\bar y - \bar y_T\|_{L^2(\I;H^1(\Omega))} + \|\bar y\|_{L^2(T_\varepsilon,\infty;H^1(\Omega))}\Big) < \varepsilon^2,
\end{align*}
which proves \eqref{E4.9}.
\end{proof}

\begin{remark}
Let us observe that $\{\bar y_T\}_{T > 0}$ is uniformly bounded in $L^\infty(Q_{T'})$ for each $T' \in (0,\infty)$. Indeed, using \eqref{E2.5}, we can derive the usual $L^\infty(Q_{T'})$ estimates for each $\bar y_T$ depending on $T'$, $\|y_0\|_{L^\infty(\Omega)}$, and $\|g + \bar u_T\chi_\omega\|_{L^\infty(0,T';L^2(\Omega)}$, which is uniformly bounded. As a consequence of this, \eqref{E2.2}, and \eqref{E2.4}, we also have the uniform boundedness of $\{f(\cdot,\cdot,\bar y_T)\}_{T > 0}$ and $\{\frac{\partial f}{\partial y}(\cdot,\cdot,\bar y_T)\}_{T > 0}$ in $L^\infty(Q_{T'})$.
\label{R4.1}
\end{remark}

\subsection{Convergence of the adjoint states}
\label{S4.2}

In this section, besides the assumptions \eqref{E2.2}--\eqref{E2.4}, we will make the following assumption
\begin{equation}
\exists\, m_f \in (0,M_f) \text{ such that } \frac{\partial f}{\partial y}(x,t,y) \ge 0 \text{ if } |y| < m_f
\label{E4.10}
\end{equation}
for almost all $(x,t) \in Q$, where $M_f$ is the constant  introduced  in \eqref{E2.3}. Let us start proving some auxiliary results before analyzing the convergence of $\{\bar\varphi_T\}_{T > 0}$.

\begin{lemma}
For every $T > 0$ let $z_T \in W(0,T)$ be the solution of the equation
\begin{equation}
\left\{\begin{array}{l}
\displaystyle
\frac{\partial z_T}{\partial t}- \Delta z_T + az_T + \frac{\partial f}{\partial y}(x,t,\bar y_T)z_T = \bar\varphi_T \mbox{ in } Q_T,\\[0.5ex] \partial_nz_T = 0 \mbox{ on } \Sigma_T,\
z_T(0) = 0 \mbox{ in } \Omega.
\end{array}\right.
\label{E4.11}
\end{equation}
Then, there exists a constant $C_z$ such that
\begin{equation}
\|z_T\|_{L^2( Q_T)} \le C_z\|\bar\varphi_T\|_{L^2(Q_T)}\quad \forall T > 0.
\label{E4.12}
\end{equation}
\label{L4.2}
\end{lemma}

\begin{proof}
It is well known that $H^1(\Omega)$ is continuously embedded in $L^4(\Omega)$ for $1 \le n \le 3$. Hence, there exists a constant $C_4$ such that $\|y\|_{L^4(\Omega)} \le C_4\|y\|_{H^1(\Omega)}$ $\forall y \in H^1(\Omega)$. From \eqref{E2.4} we deduce the existence of a constant $C_{M_f}$ such that
\begin{equation}
\big|\frac{\partial f}{\partial y}(x,t,y)\big| \le C_{M_f}\quad \forall |y| \le M_f \text{ for a.a. } (x,t) \in Q.
\label{E4.13}
\end{equation}
Applying Theorem \ref{T4.3} with $\varepsilon = \frac{m_fC_a^2}{2C_4^2C_{M_f}}$, where $C_a$ was introduced in \eqref{E2.6}, we infer the existence of $T_{a,f} > 0$ such that
\begin{equation}
\|\bar y_T(t)\|_{L^2(\Omega)} < \frac{m_fC_a^2}{2C_4^2C_{M_f}}\ \ \forall T > T_{a,f} \text{ and } \forall t > T_{a,f}.
\label{E4.14}
\end{equation}
For $T\le T_{a,f}$ inequality \eqref{E4.12} is obvious. Henceforth we consider the case $T> T_{a,f}$. We obtain in a standard manner
\begin{equation}
\|z_T\|_{L^2(Q_{T_{a,f}})} \le C\|\bar\varphi_T\|_{L^2(Q_T)} \quad \forall T > T_{a,f}
\label{E4.15}
\end{equation}
for some constant $C$ independent of $T$. For every $t \in (0,T)$ we define the sets
\[
\Omega_t^{m_f,M_f} = \{x \in \Omega : m_f \le  |\bar y_T(x,t)|  \le M_f\}.
\]
Assumptions \eqref{E2.3} and \eqref{E4.10}  imply that $\frac{\partial f}{\partial y}(x,t,\bar y_T(x,t)) \ge 0$ for $x \in \Omega \setminus \Omega_t^{m_f,M_f}$. Then, from \eqref{E2.6} and equation \eqref{E4.11} we get
\begin{align}
&C_a^2\int_0^T\|z_T\|^2_{H^1(\Omega)}\dt\notag\\
&\le \frac{1}{2}\|z_T(T)\|^2_{L^2(\Omega)} + \int_{Q_T}[|\nabla z_T|^2 + az_T^2]\dx\dt + \int_0^T\int_{\Omega \setminus \Omega_t^{m_f,M_f}}\frac{\partial f}{\partial y}(x,t,\bar y_T)z_T^2\dx\dt\notag\\
& = \int_ {Q_T}\bar\varphi_Tz_T\dx\dt - \int_0^T\int_{\Omega_t^{m_f,M_f}} \frac{\partial f}{\partial y}(x,t,\bar y_T)z_T^2\dx\dt\notag\\
& \le \|\bar\varphi_T\|_{L^2(Q_T)}\|z_T\|_{L^2(Q_T)} + \int_0^T\int_{\Omega_t^{m_f,M_f}} \big|\frac{\partial f}{\partial y}(x,t,\bar y_T)\big|z_T^2\dx\dt. \label{E4.16}
\end{align}
Let us estimate the last integral. We split the integral into two parts. Inequalities \eqref{E4.13} and \eqref{E4.15} yield
\begin{align}
&\int_0^{T_{a,f}}\int_{\Omega_t^{m_f,M_f}} \big|\frac{\partial f}{\partial y}(x,t,\bar y_T)\big|z_T^2\dx\dt\notag\\
& \le C_{M_f}\|z_T\|^2_{L^2(0,T_{a,f};L^2(\Omega))} \le C_{M_f} C^2  \|\bar\varphi_T\|^2_{L^2(Q_T)}.
\label{E4.17}
\end{align}
Now, from \eqref{E4.13}--\eqref{E4.14} we infer
\begin{align}
&\int_{T_{a,f}}^T\int_{\Omega_t^{m_f,M_f}} \big|\frac{\partial f}{\partial y}(x,t,\bar y_T)\big|z_T^2\dx\dt \le \frac{C_{M_f}}{m_f}\int_{T_{a,f}}^T\int_{\Omega_t^{m_f,M_f}} |\bar y_T|z_T^2\dx\dt\notag\\
&\le \frac{C_{M_f}}{m_f}\int_{T_{a,f}}^T\|\bar y_T\|_{L^2(\Omega)}\|z_T\|^2_{L^4(\Omega)}\dt \le \frac{C_{M_f}C_4^2}{m_f}\int_{T_{a,f}}^T\|\bar y_T\|_{L^2(\Omega)}\|z_T\|^2_{H^1(\Omega)}\dt\notag\\
&\le \frac{C_a^2}{2}\int_0^T\|z_T\|^2_{H^1(\Omega)}\dt. \label{E4.18}
\end{align}
From \eqref{E4.16}--\eqref{E4.18} we get with Young's inequality
\begin{align*}
&\frac{C_a^2}{2}\|z_T\|^2_{L^2(Q_T)} \le \frac{C_a^2}{2}\int_0^T\|z_T\|^2_{H^1(\Omega)}\dt \le \|\bar\varphi_T\|_{L^2(Q_T)}\|z_T\|_{L^2(Q_T)}\\
& + C_{M_f}C^2 \|\bar\varphi_T\|^2_{L^2(Q_T)} \le \frac{C_a^2}{4}\|z_T\|^2_{L^2(Q_T)} + \big(\frac{1}{C_a^2} + C_{M_f}C^2 \big)\|\bar\varphi_T\|^2_{L^2(Q_T)},
\end{align*}
which proves \eqref{E4.12} for $T > T_{a,f}$.
\end{proof}

\begin{lemma}
There exists a constant $C_\varphi$ such that
\begin{equation}
\|\bar\varphi_T\|_{L^2(Q)} \le C_\varphi\quad \forall T > 0.
\label{E4.19}
\end{equation}
\label{L4.3}
\end{lemma}

\begin{proof}
Let $z_T$ be as in Lemma \ref{L4.2}. From equations \eqref{E3.1} and \eqref{E4.11}, and inequality \eqref{E4.12} we infer
\begin{align*}
&\|\bar\varphi_T\|^2_{L^2(Q)} = \|\bar\varphi_T\|^2_{L^2(Q_T)} = \int_{Q_T}\Big[\frac{\partial z_T}{\partial t}- \Delta z_T + az_T + \frac{\partial f}{\partial y}(x,t,\bar y_T)z_T\Big]\bar\varphi_T\dx\dt\\
&=\int_{Q_T}\Big[-\frac{\partial\bar\varphi_T}{\partial t}- \Delta\bar\varphi_T + a\bar\varphi_T + \frac{\partial f}{\partial y}(x,t,y_T)\bar\varphi_T\Big]z_T\dx\dt\\
&= \int_{Q_T}(2\bar y_T - y_d - \bar y)z_T\dx\dt \le C_z\Big(2\|\bar y_T\|_{L^2(Q)} + \|y_d\|_{L^2(Q)} + \|\bar y\|_{L^2(Q)}\Big)\|\bar\varphi_T\|_{L^2(Q)}.
\end{align*}
Since $\{\bar y_T\}_{T > 0}$ is uniformly bounded in $L^2(Q)$, the above inequalities imply \eqref{E4.19}.
\end{proof}

Using Lemma \ref{L4.2} we infer the existence of an increasing sequence $\{T_k\}_{k = 1}^\infty$ converging to $\infty$ and a function $\bar\varphi \in L^2(Q)$ such that $\bar\varphi_{T_k} \rightharpoonup \bar\varphi$ in $L^2(Q)$ as $k \to \infty$. The next lemma  establishes stronger convergences properties of the adjoint states.

\begin{lemma}
The following convergences hold
\begin{align}
&\bar\varphi_{T_k} \stackrel{*}{\rightharpoonup} \bar\varphi \text{ in } L^2(\I;H^1(\Omega)) \cap L^\infty(\I;L^2(\Omega)),\label{E4.21}\\
&\bar\varphi_{T_k} \rightharpoonup \bar\varphi \text{ in } W(0,T)\quad\forall T \in (\I). \label{E4.22}
\end{align}
Moreover, there exists a subsequence of $\{T_k\}_{k = 1}^\infty$, denoted in the same way, such that
\begin{equation}
\lim_{k \to \infty}\|\bar\varphi_{T_k}(t)\|_{L^2(\Omega)} = \|\bar\varphi(t)\|_{L^2(\Omega)}\ \ \text{for a.a. } t \in (\I).
\label{E4.23}
\end{equation}
\label{L4.4}
\end{lemma}

\begin{proof}
Taking $\Omega_t^{m_f,M_f}$ as in the proof of Lemma \ref{L4.2}, from the adjoint state equation satisfied by $\bar\varphi_{T_k}$ and \eqref{E2.6} it follows for every  $t \in (0,T)$ with $T > 0$ arbitrary
\begin{align*}
&\frac{1}{2}\|\bar\varphi_{T_k}(t)\|_{L^2(\Omega)}^2 + C_a^2\int_t^T\|\bar\varphi_{T_k}\|^2_{H^1(\Omega)}\ds\\
&\le \frac{1}{2}\|\bar\varphi_{T_k}(t)\|_{L^2(\Omega)}^2 + \int_t^T\int_\Omega[|\nabla\bar\varphi_{T_k}|^2 + a\bar\varphi_{T_k}^2]\dx\ds + \int_t^T\int_\Omega \frac{\partial f}{\partial y}(x,t,\bar y_{T_k})\bar\varphi_{T_k}^2\dx\ds\\
& - \int_t^T\int_{\Omega_t^{m_f,M_f}}\frac{\partial f}{\partial y}(x,t,\bar y_{T_k})\bar\varphi_{T_k}^2\dx\ds\\
&\le \Big(2\|\bar y_{T_k}\|_{L^2(Q)} + \|y_d\|_{L^2(Q)} + \|\bar y\|_{L^2(Q)}\Big)\|\bar\varphi_{T_k}\|_{L^2(Q)}\\
& + \int_t^T\int_{\Omega_t^{m_f,M_f}}\big|\frac{\partial f}{\partial y}(x,t,\bar y_{T_k})\big|\bar\varphi_{T_k}^2\dx\ds.
\end{align*}
The last two terms are bounded by a constant independent of $k$. Indeed, for the first term this boundedness follows from the boundedness of $\{\bar y_{T_k}\}_{k = 1}^\infty$, see \eqref{E4.1}, and Lemma \ref{L4.3}. The boundedness of the second term is a consequence of \eqref{E4.13} and again Lemma \ref{L4.3}. Since $T > 0$ and $t \in (0,T)$ were arbitrary, and $\bar\varphi_{T_k}(x,t) = 0$ for $t > T_k$, the above inequalities imply that $\{\bar\varphi_{T_k}\}_{k =1}^\infty$ is bounded in $L^\infty(\I;L^2(\Omega))$. Additionally, taking $t \to 0$, we also infer that $\{\bar\varphi_{T_k}\}_{k =1}^\infty$ is bounded in $L^2(\I;H^1(\Omega))$. These boundedness and the convergence $\bar\varphi_{T_k} \rightharpoonup \bar\varphi$ in $L^2(Q)$ yield \eqref{E4.21}.

Using \eqref{E2.5} and the fact that the functions $\bar y_{T_k}$ are uniformly bounded with respect to $k$ in $L^\infty(Q_T)$ for every $T$, it follows from \eqref{E3.1} that $\{\bar\varphi_{T_k}\}_{k = 1}^\infty$ is bounded  in $W(0,T)$ by a constant $C_T$. Then, \eqref{E4.22} follows from \eqref{E4.21}.

Since the embedding $W(0,T) \subset L^2(Q_T)$ is compact, we have that $\bar\varphi_{T_k} \to \bar\varphi$ strongly in $L^2(Q_T)$ for every $T < \infty$. Then, we can extract a subsequence of $\{\bar\varphi_{T_k}\}_{k = 1}^\infty$, denoted by $\{\bar\varphi_{1,j}\}_{j = 1}^\infty$ such that $\|\bar\varphi_{1,j}(t)\|_{L^2(\Omega)} \to \|\bar\varphi(t)\|_{L^2(\Omega)}$ for almost all $t \in (0,1)$. In a second step, a further subsequence of $\{\bar\varphi_{1,j}\}_{j = 1}^\infty$, denoted by $\{\bar\varphi_{2,j}\}_{j = 1}^\infty$, is taken such that the pointwise convergence in time holds almost everywhere in $(0,2)$. Proceeding in this way we obtain for every $i$ a subsequence $\{\bar\varphi_{i,j}\}_{j = 1}^\infty$ such that $\|\bar\varphi_{i,j}(t)\|_{L^2(\Omega)} \to \|\bar\varphi(t)\|_{L^2(\Omega)}$ for almost all $t \in (0,i)$ when $j \to \infty$. Hence, the choice $\{\bar\varphi_{i,i}\}_{i = 1}^\infty$ satisfies \eqref{E4.23}.
\end{proof}

\begin{remark}
Let us observe that the convergence $\bar\varphi_{T_k} \rightharpoonup \bar\varphi$ in $W(0,T)$ implies that $\bar\varphi_{T_k}(t) \rightharpoonup \bar\varphi(t)$ in $L^2(\Omega)$ for every $t \in [0,T]$. Indeed, given $t \in [0,T]$, from the continuous embedding $W(0,T) \subset C([0,T];L^2(\Omega))$ we infer the continuity of the mapping $z \in W(0,T) \hookrightarrow z(t) \in L^2(\Omega)$. Therefore, if $z_k \rightharpoonup z$ in $W(0,T)$, then $z_k(t) \rightharpoonup z(t)$ in $L^2(\Omega)$.
\label{R4.2}
\end{remark}

\begin{lemma}
For every $\varepsilon > 0$ there exists $T_\varepsilon \in (\I)$ such that
\begin{equation}
\left(\|\bar\varphi_{T_k}(t)\|_{L^2(\Omega)}^2 + \int_{T_\varepsilon}^\infty\|\bar\varphi_{T_k}\|^2_{H^1(\Omega)}\dt\right)^{1/2} < \varepsilon\ \ \forall T_k > T_\varepsilon \text{ and } \forall t > T_\varepsilon.
\label{E4.24}
\end{equation}
\label{L4.5}
\end{lemma}

\begin{proof}
Given $\varepsilon > 0$, Theorem \ref{T4.3} yields the existence of $T_\varepsilon \in (\I)$  such that
\begin{equation}
\|\bar y_{T}(t)\|_{L^2(\Omega)} < \varepsilon^2\quad \forall t > T_\varepsilon\text{ and } \forall T > T_\varepsilon.
\label{E4.25}
\end{equation}
Moreover, since $\bar y_T \to \bar y$ in $L^2(Q)$ and $\bar y - y_d \in L^2(Q)$, for $T_\varepsilon$ sufficiently large we get
\[
\|\bar y_T - \bar y\|_{L^2(Q)} < \frac{\min\{1,C_a^2\}}{4C_\varphi}\varepsilon^2 \ \ \forall T > T_\varepsilon\ \text{ and }\ \|\bar y - y_d\|_{L^2(T_\varepsilon,\infty;L^2(\Omega))} < \frac{\min\{1,C_a^2\}}{4C_\varphi}\varepsilon^2,
\]
where $C_\varphi$ is the constant appearing in \eqref{E4.19}. From these estimates we infer
\begin{align}
&\|2\bar y_T - y_d - \bar y\|_{L^2(T_\varepsilon,\infty;L^2(\Omega))}\notag\\
& \le \|\bar y_T - \bar y\|_{L^2(T_\varepsilon,\infty;L^2(\Omega))} + \|\bar y - y_d\|_{L^2(T_\varepsilon,\infty;L^2(\Omega))} < \frac{\min\{1,C_a^2\}}{2C_\varphi}\varepsilon^2.
\label{E4.26}
\end{align}

Now, taking $T_\varepsilon < t < T_k$ and  proceeding similarly as in the proof of Lemma \ref{L4.4} we get with \eqref{E4.19} and \eqref{E4.26}
\begin{align}
&\frac{1}{2}\|\bar\varphi_{T_k}(t)\|_{L^2(\Omega)}^2 + C_a^2\int_t^{T_k}\|\bar\varphi_{T_k}\|^2_{H^1(\Omega)}\ds\notag\\
&\le \|2\bar y_{T_k} - y_d - \bar y\|_{L^2(T_\varepsilon,\infty;L^2(\Omega))}\|\bar\varphi_{T_k}\|_{L^2(Q)} + \int_t^{T_k}\int_{\Omega_t^{m_f,M_f}}\big|\frac{\partial f}{\partial y}(x,t,\bar y_{T_k})\big|\bar\varphi_{T_k}^2\dx\ds\notag\\
& <\frac{\min\{1,C_a^2\}}{2}\varepsilon^2 + \int_t^{T_k}\int_{\Omega_t^{m_f,M_f}}\big|\frac{\partial f}{\partial y}(x,t,\bar y_{T_k})\big|\bar\varphi_{T_k}^2\dx\ds.\label{E4.27}
\end{align}
To estimate the last term  we use \eqref{E4.25} and argue as in \eqref{E4.18} to deduce
\begin{align}
&\int_t^{T_k}\int_{\Omega_t^{m_f,M_f}} \big|\frac{\partial f}{\partial y}(x,t,\bar y_{T_k})\big|\bar\varphi_{T_k}^2\dx\dt \le \frac{C_{M_f}}{m_f}\int_t^{T_k}\int_{\Omega_t^{m_f,M_f}} |\bar y_T|\bar\varphi_{T_k}^2\dx\dt\notag\\
&\le \frac{C_{M_f}}{m_f}\int_t^{T_k}\|\bar y_{T_k}\|_{L^2(\Omega)}\|\bar\varphi_{T_k}\|^2_{L^4(\Omega)}\dt \le \frac{C_{M_f}C_4^2}{m_f}\int_t^{T_k}\|\bar y_{T_k}\|_{L^2(\Omega)}\|\bar\varphi_{T_k}\|^2_{H^1(\Omega)}\dt\notag\\
&\le \frac{C_{M_f}C_4^2}{m_f}\varepsilon^2\int_t^{T_k}\|\bar\varphi_{T_k}\|^2_{H^1(\Omega)}\dt. \label{E4.28}
\end{align}
Without loss of generality we can assume that $\frac{C_{M_f}C_4^2}{m_f}\varepsilon^2 \le \frac{\min\{1,C_a^2\}}{2}$.  Then, \eqref{E4.23} follows from \eqref{E4.27} and \eqref{E4.28}.
\end{proof}

As a consequence of this lemma we infer the following corollary.

\begin{corollary}
The following convergence holds
\begin{equation}
\lim_{t \to \infty}\|\bar\varphi(t)\|_{L^2(\Omega)} = 0.
\label{E4.29}
\end{equation}
\label{C4.1}
\end{corollary}

\begin{proof}
Given $\varepsilon > 0$, we obtain with Lemma \ref{L4.5} the existence of $T_\varepsilon > 0$ such that
\[
\|\bar\varphi_{T_k}(t)\|_{L^2(\Omega)} < \varepsilon\ \ \forall t > T_\varepsilon\text{ and } \forall T_k > T_\varepsilon.
\]
From this inequality and \eqref{E4.23}, we have that
\[
\|\bar\varphi(t)\|_{L^2(\Omega)} = \lim_{k \to \infty}\|\bar\varphi_{T_k}(t)\|_{L^2(\Omega)} \le \varepsilon\ \text{ for a.a. } t > T_\varepsilon.
\]
Since $\bar\varphi:(\I) \longrightarrow L^2(\Omega)$ is continuous, the above inequality implies $\|\bar\varphi(t)\|_{L^2(\Omega)} \le \varepsilon$ for every $t > T_\varepsilon$, which proves the corollary.
\end{proof}

\begin{lemma}
For every $t \ge 0$ the following identity holds
\begin{align}
\frac{1}{2}\|\bar\varphi(t)\|_{L^2(\Omega)} &+ \int_t^\infty\int_\Omega[|\nabla\bar\varphi|^2 + a\bar\varphi^2]\dx\ds + \int_t^\infty\int_\Omega\frac{\partial f}{\partial y}(x,t,\bar y)\bar\varphi^2\dx\dt\notag\\
& = \int_t^\infty\int_\Omega(\bar y - y_d)\bar\varphi\dx\dt. \label{E4.30}
\end{align}
\label{L4.6}
\end{lemma}

\begin{proof}
We split the proof into two steps.

{\em Step 1 - $\big|\frac{\partial f}{\partial y}(\cdot,\cdot,\bar y_{T_k})\big|^{1/2}\bar\varphi_{T_k} \rightharpoonup \big|\frac{\partial f}{\partial y}(\cdot,\cdot,\bar y)\big|^{1/2}\bar\varphi$ in $L^2(Q)$.} Let us first prove the boundedness of $\big\{|\frac{\partial f}{\partial y}(\cdot,\cdot,\bar y_{T_k})\big|^{1/2}\bar\varphi_{T_k}\}_{k = 1}^\infty$ in $L^2(Q)$. We use \eqref{E3.1}, \eqref{E2.4}, and \eqref{E4.10} to get
\begin{align*}
&\int_{Q}\big|\frac{\partial f}{\partial y}(\cdot,\cdot,\bar y_{T_k})\big|\bar\varphi_{T_k}^2\dx\dt = \int_{Q_{T_k}}\big|\frac{\partial f}{\partial y}(\cdot,\cdot,\bar y_{T_k})\big|\bar\varphi_{T_k}^2\dx\dt\\
&= \int_0^{T_k}\int_{\Omega \setminus \Omega_t^{m_f,M_f}}\frac{\partial f}{\partial y}(\cdot,\cdot,\bar y_{T_k})\bar\varphi_{T_k}^2\dx\dt + \int_0^{T_k}\int_{\Omega_t^{m_f,M_f}}\big|\frac{\partial f}{\partial y}(\cdot,\cdot,\bar y_{T_k})\big|\bar\varphi_{T_k}^2\dx\dt\\
&\le \int_{Q_{T_k}}(2\bar y_{T_k} - y_d - \bar y)\bar\varphi_{T_k}\dx\dt +2\int_0^{T_k}\int_{\Omega_t^{m_f,M_f}}\big|\frac{\partial f}{\partial y}(\cdot,\cdot,\bar y_{T_k})\big|\bar\varphi_{T_k}^2\dx\dt\\
&\|2\bar y_{T_k} - y_d - \bar y\|_{L^2(Q)}\|\bar\varphi_{T_k}\|_{L^2(Q)} + C_{M_f}\|\bar\varphi_{T_k}\|^{2}_{L^2(Q)}.
\end{align*}
Using the boundedness of $\{\bar\varphi_{T_k}\}_{k=1}^\infty$ and $\{\bar y_{T_k}\}_{k=1}^\infty$ in $L^2(Q)$ we deduce the desired boundedness. Then, there exist subsequences, denoting in the same way, such that $\big|\frac{\partial f}{\partial y}(\cdot,\cdot,\bar y_{T_k})\big|^{1/2}\bar\varphi_{T_k} \rightharpoonup \psi$ in $L^2(Q)$. From \eqref{E4.22} and the compactness of the embedding $W(0,T) \subset L^2(Q_T)$ we get the strong convergence $\bar\varphi_{T_k} \to \bar\varphi$ in $L^2(Q_T)$ for every $T < \infty$. Then, taking subsequences we can assume that $\bar y_{T_k}(x,t) \to \bar y(x,t)$ and $\bar\varphi_{T_k}(x,t) \to \bar\varphi(x,t)$ for almost all $(x,t) \in Q_T$. This implies that
\[
\big|\frac{\partial f}{\partial y}(x,t,\bar y_{T_k}(x,t))\big|^{1/2}\bar\varphi_{T_k}(x,t) \to \big|\frac{\partial f}{\partial y}(x,t,\bar y(x,t))\big|^{1/2}\bar\varphi(x,t) \text{ for a.a. } (x,t) \in Q_T,
\]
and, consequently $\psi(x,t) = \big|\frac{\partial f}{\partial y}(x,t,\bar y(x,t))\big|^{1/2}\bar\varphi(x,t)$ in $Q_T$. As $T$ was arbitrary we infer that $\big|\frac{\partial f}{\partial y}(\cdot,\cdot,\bar y)\big|^{1/2}\bar\varphi = \psi$. Since all subsequences have the same limit, the whole sequence converges to the claimed limit.

{\em Step 2 - Proof of \eqref{E4.30}.} Given $z \in L^2(\IT;H^1(\Omega))$, we deduce from \eqref{E3.1}
\begin{align*}
&-\int_t^T\langle\frac{\partial\bar\varphi_{T_k}}{\partial t},z\rangle\ds + \int_t^T\int_\Omega[\nabla\bar\varphi_{T_k}\nabla z + a\bar\varphi_{T_k}z]\dx\ds\\
&  + \int_t^T\int_\Omega\frac{\partial f}{\partial y}(x,t,\bar y_{T_k})\bar\varphi_{T_k}z\dx\ds = \int_t^T\int_\Omega(2\bar y_{T_k} - y_d - \bar y)z\dx\ds
\end{align*}
for every $T > 0$, $T_k > T$, and $t \in [0,T)$. Using \eqref{E4.1}, \eqref{E4.21}, and \eqref{E4.22} we can pass to the limit in the above identity and obtain
\begin{align}
&-\int_t^T\langle\frac{\partial\bar\varphi}{\partial t},z\rangle\ds + \int_t^T\int_\Omega[\nabla\bar\varphi\nabla z + a\bar\varphi z]\dx\ds + \int_t^T\int_\Omega\frac{\partial f}{\partial y}(x,t,\bar y)\bar\varphi z\dx\ds\notag\\
& = \int_t^T\int_\Omega(\bar y - y_d)z\dx\ds\quad \forall t \in [0,T]. \label{E4.31}
\end{align}
The only limit which is not obvious is
\[
\lim_{k \to \infty}\int_t^T\int_\Omega\frac{\partial f}{\partial y}(x,t,\bar y_{T_k})\bar\varphi_{T_k}z\dx\ds = \int_t^T\int_\Omega\frac{\partial f}{\partial y}(x,t,\bar y)\bar\varphi z\dx\ds.
\]
It follows from Step 1 and the fact that
\[
\big|\frac{\partial f}{\partial y}(\cdot,\cdot,\bar y_{T_k})\big|^{1/2}\sign{\frac{\partial f}{\partial y}(\cdot,\cdot,\bar y_{T_k})}z \to \big|\frac{\partial f}{\partial y}(\cdot,\cdot,\bar y)\big|^{1/2}\sign{\frac{\partial f}{\partial y}(\cdot,\cdot,\bar y)}z \text{ in }L^2(Q).
\]
This last convergence can be easily deduced taking into account Lebesgue's dominated convergence theorem and Remark \ref{R4.1}. Now, taking $z = \bar\varphi$ in \eqref{E4.31} and recalling that $\bar\varphi \in W(0,T)$ for every $T < \infty$, we get
\begin{align*}
\frac{1}{2}\|\bar\varphi(t)\|^2_{L^2(\Omega)} &+ \int_t^T\int_\Omega[|\nabla\bar\varphi|^2 + a\bar\varphi^2]\dx\ds + \int_t^T\int_\Omega\frac{\partial f}{\partial y}(x,t,\bar y)\bar\varphi^2\dx\dt\notag\\
& = \int_t^T\int_\Omega(\bar y - y_d)\bar\varphi\dx\dt + \frac{1}{2}\|\bar\varphi(T)\|^2_{L^2(\Omega)}.
\end{align*}
Finally, taking $T \to \infty$ and using \eqref{E4.29} identity \eqref{E4.30} follows.
\end{proof}

\begin{lemma}
Strong convergence $\bar\varphi_{T_k} \to \bar\varphi$ in $L^2(\I;H^1(\Omega))$ as $k \to \infty$ holds.
\label{L4.7}
\end{lemma}

\begin{proof}
Recalling Remark \ref{R4.2} and using the convergence established in Step 1 of the proof of Lemma \ref{L4.6} we get with \eqref{E3.1} and \eqref{E4.30}
\begin{align*}
&\frac{1}{2}\|\bar\varphi(0)\|^2_{L^2(\Omega)} + \int_0^\infty\int_\Omega[|\nabla\bar\varphi|^2 + a\bar\varphi^2]\dx\dt\\
& \le \liminf_{k \to \infty}\left\{\frac{1}{2}\|\bar\varphi_{T_k}(0)\|^2_{L^2(\Omega)} + \int_0^\infty\int_\Omega[|\nabla\bar\varphi_{T_k}^2 + a\bar\varphi_{T_k}^2]\dx\dt\right\}\\
&\le \limsup_{k \to \infty}\left\{\frac{1}{2}\|\bar\varphi_{T_k}(0)\|^2_{L^2(\Omega)} + \int^\infty_0\int_\Omega[|\nabla\bar\varphi_{T_k}^2 + a\bar\varphi_{T_k}^2]\dx\dt\right\}\\
& = \limsup_{k \to \infty}\left\{\frac{1}{2}\|\bar\varphi_{T_k}(0)\|^2_{L^2(\Omega)} + \int_0^{T_k}\int_\Omega[|\nabla\bar\varphi_{T_k}^2 + a\bar\varphi_{T_k}^2]\dx\dt\right\}\\
&= \limsup_{k \to \infty}\left\{\int_0^\infty\int_\Omega(2\bar y_{T_k} - y_d - \bar y)\bar\varphi_{T_k}\dx\dt - \int_0^\infty\int_\Omega\frac{\partial f}{\partial y}(x,t,\bar y_{T_k})\bar\varphi_{T_k}^2\dx\dt\right\}\\
&\le \int_0^\infty\int_\Omega(\bar y - y_d)\bar\varphi\dx\dt -  \liminf_{k \to \infty}\int_0^\infty\int_\Omega\frac{\partial f}{\partial y}(x,t,\bar y_{T_k})\bar\varphi_{T_k}^2\dx\dt\\
& \le \frac{1}{2}\|\bar\varphi(0)\|^2_{L^2(\Omega)} + \int_0^\infty\int_\Omega[|\nabla\bar\varphi|^2 + a\bar\varphi^2]\dx\dt.
\end{align*}
Recalling Lemma \ref{L4.1}, the above inequalities imply the convergence
\[
\lim_{k \to \infty} \int_0^\infty\int_\Omega[|\nabla\bar\varphi_{T_k}^2 + a\bar\varphi_{T_k}^2]\dx\dt = \int_0^\infty\int_\Omega[|\nabla\bar\varphi|^2 + a\bar\varphi^2]\dx\dt.
\]
This identity and the weak convergence $\bar\varphi_{T_k} \rightharpoonup \bar\varphi$ in $L^2(0,\infty;H^1(\Omega))$ prove the strong convergence.
\end{proof}

Analogously to Definition \ref{D2.1} we have the following definition.
\begin{definition}
We call $\varphi$ a solution to
\begin{equation}
\left\{\begin{array}{rcll}
\displaystyle
-\frac{\partial\varphi}{\partial t}- \Delta\varphi + a\varphi + \frac{\partial f}{\partial y}(x,t,\bar y)\varphi& = &\displaystyle \bar y - y_d&
\mbox{in } Q,\\[0.5ex]
\partial_n\varphi & = & 0      & \mbox{on } \Sigma\end{array}\right.
\label{E4.32}
\end{equation}
if $\varphi \in L^2(\I;H^1(\Omega))$ and for every $T > 0$ the restriction of $\varphi$ to $Q_T$ belongs to $W(0,T)$ and satisfies
\begin{align}
&-\int_0^T\langle\frac{\partial\varphi}{\partial t},z\rangle\dt + \int_0^T\int_\Omega[\nabla\varphi\nabla z + a\varphi z]\dx\dt + \int_0^T\int_\Omega\frac{\partial f}{\partial y}(x,t,\bar y)\varphi z\dx\dt\notag\\
& = \int_0^T\int_\Omega(\bar y - y_d)z\dx\dt \quad \forall z \in L^2(\IT;H^1(\Omega)), \label{E4.33}\\
&\lim_{t \to \infty}\|\varphi(t)\|_{L^2(\Omega)} = 0. \label{E4.34}
\end{align}
\label{D4.1}
\end{definition}

\begin{theorem}
The function $\bar\varphi$ is the unique solution of \eqref{E4.32} and $\bar\varphi_T \to \bar\varphi$ strongly in $L^2(\I;H^1(\Omega))$ as $T \to \infty$.
\label{T4.4}
\end{theorem}

\begin{proof}
The fact that $\bar\varphi$ is a solution of \eqref{E4.32} follows from \eqref{E4.29} and \eqref{E4.31}. Let us prove the uniqueness. It is enough to prove that the unique function satisfying \eqref{E4.33} and \eqref{E4.34} with a zero right hand side in \eqref{E4.33} is the zero function. From \eqref{E2.9} we deduce the existence of $T_{a,f} < \infty$ such that
\[
\|\bar y(t)\|_{L^2(\Omega)} < \frac{m_fC_a^2}{2C_4^2C_{M_f}}\ \ \ \forall t > T_{a,f}.
\]
Using this inequality and arguing as in \eqref{E4.18}
we infer
\begin{equation}
\int_{T_{a,f}}^T\int_{\Omega_t^{m_f,M_f}} \big|\frac{\partial f}{\partial y}(x,t,\bar y)\big|\varphi^2\dx\dt \le \frac{C_a^2}{2}\int_0^T\|
\varphi\|^2_{H^1(\Omega)}\dt \ \ \forall T > T_{a,f}.
\label{E4.35}
\end{equation}
Now, we take
\[
z(x,t) = \left\{\begin{array}{cl}\displaystyle \text{\rm e}^{2C_{M_f}(t - T_{a,f})}\varphi(x,t) & \text{if } t \le T_{a,f},\\\varphi(x,t) & \text{otherwise.}\end{array}\right.
\]
Inserting this function in \eqref{E4.33} we obtain for every $T > T_{a,f}$
\begin{align}
&-\frac{1}{2}\int_0^{T_{a,f}}\text{\rm e}^{2C_{M_f}(t - T_{a,f})}\frac{d}{dt}\|\varphi(t)\|^2_{L^2(\Omega)}\dt -\frac{1}{2}\int_{T_{a,f}}^T\frac{d}{dt}\|\varphi(t)\|^2_{L^2(\Omega)}\dt\notag\\
&  + \int_0^{T_{a,f}}\int_\Omega\text{\rm e}^{2C_{M_f}(t - T_{a,f})}[|\nabla\varphi|^2 + a\varphi^2]\dx \dt +  \int_{T_{a,f}}^T\int_\Omega[|\nabla\varphi|^2 + a\varphi^2]\dx \dt\label{E4.36}\\
& + \int_0^{T_{a,f}}\int_\Omega\text{\rm e}^{2C_{M_f}(t - T_{a,f})}\frac{\partial f}{\partial y}(x,t,\bar y)\varphi^2\dx\dt + \int_{T_{a,f}}^T\int_\Omega\frac{\partial f}{\partial y}(x,t,\bar y)\varphi^2\dx\dt= 0.\notag
\end{align}
Integrating by parts and using \eqref{E2.5} we deduce
\begin{align}
&-\frac{1}{2}\int_0^{T_{a,f}}\text{\rm e}^{2C_{M_f}(t - T_{a,f})}\frac{d}{dt}\|\varphi(t)\|^2_{L^2(\Omega)}\dt + \int_0^{T_{a,f}}\int_\Omega\text{\rm e}^{2C_{M_f}(t - T_{a,f})}\frac{\partial f}{\partial y}(x,t,\bar y)\varphi^2\dx\dt\notag\\
& \ge  \frac{\exp{(-2C_{M_f}T_{a,f})}}{2}\|\varphi(0)\|^2_{L^2(\Omega)} - \frac{1}{2}\|\varphi(T_{a,f})\|^2_{L^2(\Omega)}.\label{E4.37}
\end{align}
We also have with \eqref{E2.6}
\begin{align}
&\hspace{-0.5cm}-\frac{1}{2}\int_{T_{a,f}}^T\frac{d}{dt}\|\varphi(t)\|^2_{L^2(\Omega)}\dt = \frac{1}{2}\|\varphi(T_{a,f})\|^2_{L^2(\Omega)} - \frac{1}{2}\|\varphi(T)\|^2_{L^2(\Omega)},\label{E4.38}\\
&\hspace{-0.5cm}\int_0^{T_{a,f}}\int_\Omega\text{\rm e}^{2C_{M_f}(t - T_{a,f})}[|\nabla\varphi|^2 + a\varphi^2]\dx \dt\ge \text{\rm e}^{-2C_{M_f}T_{a,f}}C_a^2\int_0^{T_{a,f}}\|\varphi\|^2_{H^1(\Omega)}\dt.\label{E4.39}
\end{align}
Finally, from \eqref{E2.6},  \eqref{E4.10},  and and \eqref{E4.35} we get
\begin{equation}
\int_{T_{a,f}}^T\int_\Omega[|\nabla\varphi|^2 + a\varphi^2]\dx \dt + \int_{T_{a,f}}^T\int_\Omega\frac{\partial f}{\partial y}(x,t,\bar y)\varphi^2\dx\dt \ge \frac{C_a^2}{2}\int_{T_{a,f}}^T\|\varphi\|^2_{H^1(\Omega)}\dt. \label{E4.40}
\end{equation}
Adding the relationships \eqref{E4.37}--\eqref{E4.40} we obtain with \eqref{E4.36}
\[
 \frac{\exp{(-2C_{M_f}T_{a,f})}}{2}\Big[\|\varphi(0)\|^2_{L^2(\Omega)} + C_a^2\int_0^T\|\varphi\|^2_{H^1(\Omega)}\dt\Big] \le \frac{1}{2}\|\varphi(T)\|^2_{L^2(\Omega)}.
\]
Taking $T \to \infty$ and using \eqref{E4.34} we conclude that $\|\varphi\|_{L^2(\I;H^1(\Omega))} = 0$ and the uniqueness follows. This also implies the uniqueness of the limits of subsequences and, hence with Lemma \ref{L4.7}, the whole family $\{\bar\varphi_T\}_{T > 0}$ converges to $\bar\varphi$ in $L^2(\I;H^1(\Omega))$ as $T \to \infty$.
\end{proof}

\subsection{Convergence of $\{(\bar\lambda_T,\bar\mu_T)\}_{T > 0}$}
\label{S4.3}
The aim of this section is to prove the following theorem.

\begin{theorem}
The family $\{(\bar\lambda_T,\bar\mu_T)\}_{T > 0}$ is bounded in $L^\infty(\I;L^2(\Omega))^2$. Therefore, there exist sequences $\{T_k\}_{k = 1}^\infty$ converging to $\infty$ such that $(\bar\lambda_{T_k},\bar\mu_{T_k}) \stackrel{*}{\rightharpoonup} (\bar\lambda,\bar\mu)$ in $L^\infty(\I;L^2(\Omega))^2$ as $k \to \infty$ holds. Moreover, each of these limits satisfies
\begin{align}
&\bar\lambda \in \partial j(\bar u),\label{E4.41}\\
&\int_0^\infty\int_\omega\bar\mu(u - \bar u)\dx\dt \le 0\quad \forall u \in \Uad, \label{E4.42}\\
&{\bar\varphi}_{\mid Q_\omega} + \kappa\bar\lambda + \bar\mu = 0.\label{E4.43}
\end{align}
In addition, if $\K$ is given by \eqref{E1.2} or \eqref{E1.3}, then $(\bar\lambda,\bar\mu)$ is unique and $(\bar\lambda_{T},\bar\mu_{T}) \stackrel{*}{\rightharpoonup} (\bar\lambda,\bar\mu)$ in $L^\infty(\I;L^2(\Omega))$ as $T \to \infty$.
\label{T4.5}
\end{theorem}

\begin{proof}
The boundedness of  $\{({\bar\varphi}_{T\mid Q_\omega},\bar\lambda_T)\}_{T > 0}$ in $L^\infty(\I;L^2(\Omega))^2$  follows from \eqref{E3.5} and \eqref{E4.21}. This along with the identity \eqref{E3.3} yields the boundedness of $\{\bar\mu_T\}_{T > 0}$. Hence, there exists a subsequence $\{T_k\}_{k = 1}^\infty$ converging to $\infty$ such that $\{(\bar\lambda_{T_k},\bar\mu_{T_k})\}_{k = 1}^\infty$ converges weakly$^*$ in $L^\infty(\I;L^2(\Omega))^2$ to elements $(\bar\lambda,\bar\mu)$ as $k \to \infty$. Then, \eqref{E4.43} is obtained passing to the limit in \eqref{E3.3}. Let us prove that $\bar\lambda$ and $\bar\mu$ satisfy \eqref{E4.41} and \eqref{E4.42}, respectively. Since $\bar\lambda_{T_k} \in \partial j_{T_k}(\bar u_{T_k})$, for every $u \in L^1(\I;L^2(\Omega))$ we have
\begin{align*}
&\int_0^\infty\int_\omega \bar\lambda_{T_k}(u - \bar u_{T_k})\dx\dt + j(\bar u_{T_k})\\
& = \int_0^{T_k}\int_\omega \bar\lambda_{T_k}(u - \bar u_{T_k})\dx\dt + j_{T_k}(\bar u_{T_k}) \le j_{T_k}(u) \le j(u).
\end{align*}
Weak$^*$ convergence $\bar\lambda_{T_k} \stackrel{*}{\rightharpoonup}\bar\lambda$  in $L^\infty(\I;L^2(\Omega))$  and \eqref{E4.3} yield
\[
\lim_{k \to \infty}\int_0^\infty\int_\omega\bar\lambda_{T_k}u\dx\dt = \int_0^\infty\int_\omega\bar\lambda u\dx\dt \ \text{ and }\ j(\bar u) = \lim_{k \to \infty}j(\bar u_{T_k}).
\]
To establish \eqref{E4.41} we shall verify
\begin{equation}
\lim_{k \to \infty}\int_0^\infty\int_\omega\bar\lambda_{T_k}\bar u_{T_k} \dx\dt = \int_0^\infty\int_\omega\bar\lambda\bar u\dx\dt
\label{E4.44}
\end{equation}
below. Inequality \eqref{E4.42} is  obtained passing to limit in \eqref{E3.2} assuming  that the convergence
\begin{equation}
\lim_{k \to \infty}\int_0^\infty\int_\omega\bar\mu_{T_k}\bar u_{T_k} \dx\dt = \int_0^\infty\int_\omega\bar\mu\bar u\dx\dt
\label{E4.45}
\end{equation}
holds. To prove \eqref{E4.44} and \eqref{E4.45} we use Lemma \ref{L4.1}. For this purpose we take into account  the strong convergence $\bar\varphi_{T_k} \to \bar\varphi$ in $L^2(Q)$, \eqref{E3.3}, \eqref{E4.2}, and \eqref{E4.43} to get
\begin{align}
&\lim_{k \to \infty}\Big(\kappa\int_0^\infty\int_\omega\bar\lambda_{T_k}\bar u_{T_k}\dx\dt +  \int_0^\infty\int_\omega\bar\mu_{T_k}\bar u_{T_k}\dx\dt\Big)\notag\\
&= -\lim_{k \to \infty}\int_0^\infty\int_\omega\bar\varphi_{T_k}\bar u_{T_k}\dx\dt = -\int_0^\infty\int_\omega\bar\varphi\bar u\dx\dt\notag\\
& = \kappa\int_0^\infty\int_\omega\bar\lambda\bar u\dx\dt +  \int_0^\infty\int_\omega\bar\mu\bar u\dx\dt.\label{E4.46}
\end{align}
From \eqref{E3.2} we infer
\[
\int_0^\infty\int_\omega\bar\mu u \dx\dt = \lim_{k \to \infty}\int_0^\infty\bar\mu_{T_k}u\dx\dt \le \liminf_{k \to \infty}\int_0^\infty\bar\mu_{T_k}\bar u_{T_k}\dx\dt \ \ \forall u \in \Uad.
\]
Taking $u = \bar u$, we obtain
\begin{equation}
\int_0^\infty\int_\omega\bar\mu\bar u \dx\dt \le \liminf_{k \to \infty}\int_0^\infty\bar\mu_{T_k}\bar u_{T_k}\dx\dt.\label{E4.47}
\end{equation}
Further, from \eqref{E3.5} we get $\|\bar\lambda\|_{L^\infty(\I;L^2(\omega))} \le \liminf_{k \to \infty}\|\bar\lambda_{T_k}\|_{L^\infty(\I;L^2(\omega))} \le 1$. Using this fact, \eqref{E4.3}, and again \eqref{E3.5} we deduce
\begin{align}
&\int_0^\infty\int_\omega\bar\lambda\bar u \dx\dt \le \|\bar u\|_{L^1(\I;L^2(\omega))}\notag\\
& = \lim_{k \to \infty}\|\bar u_{T_k}\|_{L^1(\I;L^2(\omega))} = \lim_{k \to \infty}\int_0^\infty\int_\omega\bar\lambda_{T_k}\bar u_{T_k} \dx\dt.
\label{E4.48}
\end{align}
Then, \eqref{E4.46}, \eqref{E4.47}, \eqref{E4.48}, and Lemma \ref{L4.1} imply \eqref{E4.44} and \eqref{E4.45}. Thus, \eqref{E4.41}--\eqref{E4.43} are satisfied by $(\bar u,{\bar\varphi}_{\mid_\omega},\bar\lambda,\bar\mu)$.

Now, we assume that $\K$ is given by \eqref{E1.2} or \eqref{E1.3}. From $\bar\lambda \in \partial j(\bar u)$ we know that \cite[Proposition 3.8]{CHW2017}
\[
\bar\lambda(x,t) = \frac{\bar u(x,t)}{\|\bar u(t)\|_{L^2(\omega)}} \text{ for a.a. } (x,t) \in Q_\omega \text{ if } \|\bar u(t)\|_{L^2(\omega)} \neq 0.
\]
Furthermore, from \eqref{E4.42}, similarly to \eqref{E3.7} and \eqref{E3.9}, we infer that $\|\bar\mu(t)\|_{L^2(\omega)} = 0$ if $\|\bar u(t)\|_{L^2(\omega)} = 0$. These facts along with \eqref{E4.43} lead to
\begin{equation}
\bar\lambda(x,t) = \left\{\begin{array}{cl}\displaystyle - \frac{1}{\kappa}\bar\varphi_{\mid Q_\omega}(x,t)&\text{if }\|\bar u(t)\|_{L^2(\omega)} = 0,\\[1.5ex]\displaystyle\frac{\bar u(x,t)}{\|\bar u(t)\|_{L^2(\omega)}}&\text{otherwise,}\end{array}\right. \text{ for a.a. } (x,t) \in Q_\omega.
\label{E4.49}
\end{equation}
Therefore, the limit $\bar\lambda$ is uniquely defined. Consequently, the whole family $\{\bar\lambda_T\}_{T > 0}$ converges to $\bar\lambda$. Using again \eqref{E4.43}, we deduce that the whole family $\{\bar\mu_T\}_{T > 0}$ converges to $-(\bar\varphi_{\mid Q_\omega} + \kappa\bar\lambda) = \bar\mu$. This concludes the proof.
\end{proof}

\section{Optimality Conditions for Problem \Pb}
\label{S5}\setcounter{equation}{0}

The following theorem is a consequence of Lemma \ref{L4.4}  and Theorems \ref{T4.4} and \ref{T4.5}.

\begin{theorem}
If $\bar u$ is a solution of \Pb with associated state $\bar y$, then there exist $\bar\varphi \in L^2(\I;H^1(\Omega)) \cap L^\infty(\I;L^2(\Omega))$ such that ${\bar\varphi}_{\mid_{Q_T}} \in W(0,T)$ for every $T < \infty$, $\bar\lambda \in \partial j(\bar u) \subset L^\infty(\I;L^2(\omega))$, and $\bar\mu \in L^\infty(\I;L^2(\omega))$ satisfying
\begin{align}
&\left\{\begin{array}{rcll}
\displaystyle
-\frac{\partial\bar\varphi}{\partial t}- \Delta\bar\varphi + a\varphi + \frac{\partial f}{\partial y}(x,t,\bar y)\bar\varphi& = &\displaystyle \bar y - y_d&
\mbox{in } Q,\\[0.5ex]
\partial_n\bar\varphi & = & 0      & \mbox{on } \Sigma\end{array}\right.
\label{E5.1}\\
&\int_0^\infty\int_\omega\bar\mu(u - \bar u)\dx\dt \le 0\quad \forall u \in \Uad, \label{E5.2}\\
&{\bar\varphi}_{\mid Q_\omega} + \kappa\bar\lambda + \bar\mu = 0.\label{E5.3}
\end{align}
The adjoint state $\bar\varphi$ is unique. Moreover, if $\K$ is given by \eqref{E1.2} or \eqref{E1.3}, then $\bar\lambda$ and $\bar\mu$ are unique as well.
\label{T5.1}
\end{theorem}

The function $\bar\lambda$ satisfies \eqref{E4.49}. Moreover, if $\K$ is given by \eqref{E1.2} or \eqref{E1.3}, then $(\bar\mu,\bar u)$ satisfy \eqref{E3.6}-\eqref{E3.8} or \eqref{E3.9}, respectively,  with $(\mu_T,u_T)$ replaced by $(\bar\mu,\bar u)$ and for almost all $T \in (0,\infty)$. In particular, these properties imply that $\sign{\bar\lambda(x,t)} = \sign{\bar\mu(x,t)} = \sign{\bar u(x,t)}$ for almost all $(x,t) \in Q_\omega$. Therefore,  we have the inequality
\begin{equation}
\|\kappa\bar\lambda(t) + \bar\mu(t)\|_{L^2(\omega)} \ge \kappa\|\bar\lambda(t)\|_{L^2(\omega)}\ \text{ for a.a. } t \in (0,\infty).
\label{E5.4}
\end{equation}
Using it, we deduce the following consequence from Theorem \ref{T5.1}.

\begin{corollary}
Let $\K$ be given by \eqref{E1.2} or \eqref{E1.3}. Then,  the following sparsity property holds for almost all $t \in (0,\infty)$
\begin{equation}
\left\{\begin{array}{ll}\displaystyle\text{if } \|\bar\varphi(t)\|_{L^2(\omega)} < \kappa&\Rightarrow \|\bar u(t)\|_{L^2(\omega)} = 0,\\\displaystyle\text{if } \|\bar u(t)\|_{L^2(\omega)} = 0&\Rightarrow \|\bar\varphi(t)\|_{L^2(\omega)} \le \kappa.\end{array}\right.
\label{E5.5}
\end{equation}
\label{C5.1}
\end{corollary}

\begin{proof}
Let us assume that $\|\bar u(t)\|_{L^2(\omega)} \neq 0$, then \eqref{E4.49} implies that $\|\bar\lambda(t)\|_{L^2(\omega)} = 1$. From this fact, \eqref{E5.3}, and \eqref{E5.4} we infer
\[
\|\bar\varphi(t)\|_{L^2(\omega)} \ge \kappa\|\bar\lambda(t)\|_{L^2(\omega)} = \kappa.
\]
This proves the first implication of \eqref{E5.5}. Assume now that $\|\bar u(t)\|_{L^2(\omega)} = 0$. As mentioned above, this implies that $\|\bar\mu(t)\|_{L^2(\omega)} = 0$. Hence, from \eqref{E5.3} and \eqref{E4.49} we deduce that $\|\bar\varphi(t)\|_{L^2(\omega)} = \frac{1}{\kappa}\|\bar\lambda(t)\|_{L^2(\omega)} \le \kappa$.
\end{proof}

\begin{corollary}
Let $\K$ be given by \eqref{E1.2} or \eqref{E1.3}. Then, there exists $T_0 < \infty$ such that $\bar u(x,t) = 0$ for almost all $x \in \omega$ and $t \ge T_0$.
\label{C5.2}
\end{corollary}

This corollary is a straightforward consequence of \eqref{E5.5} and \eqref{E4.29}. This shows that the optimal control shuts down to zero in finite time.  It is due to the appearance of the non-smooth term in the cost functional.

%\bibliographystyle{plain}
%\bibliography{Casas-Kunisch}
\end{document}